\documentclass[10pt, a4paper, reqno]{amsart}
\usepackage[utf8]{inputenc}
\usepackage{mathrsfs}
\usepackage{bm}
\usepackage{amstext}
\usepackage{amsthm}
\usepackage{amssymb}
\usepackage{mathtools}
\usepackage[all]{xy}
\usepackage{cleveref}
\usepackage{graphics}
\makeatletter
\numberwithin{equation}{section}
\numberwithin{figure}{section}
\theoremstyle{plain}
\newtheorem{thm}{Theorem}[section]
  \crefname{thm}{Theorem}{Theorems}
  
  \crefname{lem}{Lemma}{Lemmas}
  \newtheorem{prop}[thm]{Proposition}
  \crefname{prop}{Proposition}{Propositions}
  \newtheorem{cor}[thm]{Corollary}
	\crefname{cor}{Corollary}{Corollaries}
  \newtheorem*{theor1}{Theorem 1}
  \newtheorem*{theor2}{Theorem 2}
  \newtheorem*{theor3}{Theorem 3}
  
  \newtheorem*{ack}{Acknowledgments}
\theoremstyle{definition}
  \newtheorem{defi}[thm]{Definition}
  \crefname{defi}{Definition}{Definitions}
  \theoremstyle{remark}
  \newtheorem{ntn}[thm]{Notation}
  \crefname{ntn}{Notation}{Notations}
	 \theoremstyle{remark}
  \newtheorem{rem}[thm]{Remark}
  \crefname{rem}{Remark}{Remarks}
  \newtheorem{ex}[thm]{Example}
  \crefname{ex}{Example}{Examples}
\makeatother
\usepackage{a4wide}
\def\r{\mathbb{R}}
\def\c{\mathbb{C}}

\def\z{\mathbb{Z}}
\makeatletter
    
    \@addtoreset{equation}{section}
 \makeatother

\title[Newton-Okounkov bodies of flag varieties and combinatorial mutations]{\fontsize{11pt}{11pt}\selectfont Newton-Okounkov bodies of flag varieties and combinatorial mutations}

\date{}

\author{Naoki FUJITA}

\address[Naoki FUJITA]{Graduate School of Mathematical Sciences, The University of Tokyo, 3-8-1 Komaba, Meguro-ku, Tokyo 153-8914, Japan.}

\email{nfujita@ms.u-tokyo.ac.jp}

\author{Akihiro HIGASHITANI}

\address[Akihiro HIGASHITANI]{Department of Pure and Applied Mathematics, Graduate School of Information Science and Technology, Osaka University, Suita, Osaka 565-0871, Japan}

\email{higashitani@ist.osaka-u.ac.jp}

\subjclass[2010]{Primary 05E10; Secondary 13F60, 14M15, 14M25, 52B20}

\keywords{Newton--Okounkov body, combinatorial mutation, flag variety, cluster algebra, tropicalized mutation}

\thanks{The work of the first named author was supported by Grant-in-Aid for JSPS Fellows (No.\ 19J00123), and the work of the second named author was supported by Grant-in-Aid for Young Scientists (B) 17K14177.} 

\begin{document}
\begin{abstract}
A Newton--Okounkov body is a convex body constructed from a projective variety with a globally generated line bundle and with a higher rank valuation on the function field, which gives a systematic method of constructing toric degenerations of projective varieties. Its combinatorial properties heavily depend on the choice of a valuation, and it is a fundamental problem to relate Newton--Okounkov bodies associated with different kinds of valuations. In this paper, we address this problem for flag varieties using the framework of combinatorial mutations which was introduced in the context of mirror symmetry for Fano manifolds. By applying iterated combinatorial mutations, we connect specific Newton--Okounkov bodies of flag varieties including string polytopes, Nakashima--Zelevinsky polytopes, and FFLV polytopes.
\end{abstract}

\maketitle
\tableofcontents 

\section{Introduction}\label{s:intro}

A Newton--Okounkov body $\Delta(X, \mathcal{L}, v)$ is a convex body defined from a projective variety $X$ with a globally generated line bundle $\mathcal{L}$ on $X$ and with a higher rank valuation $v$ on the function field $\mathbb{C}(X)$, which generalizes the notion of Newton polytopes for toric varieties to other projective varieties. 
It was introduced by Okounkov \cite{Oko1, Oko2, Oko3} and developed independently by Lazarsfeld--Mustata \cite{LM} and by Kaveh--Khovanskii \cite{KK1}. 
A remarkable application is that the theory of Newton--Okounkov bodies can be used to construct toric degenerations (i.e.\ flat degenerations to toric varieties) and completely integrable systems (see \cite{And, HK}). 
Since combinatorial properties of $\Delta(X, \mathcal{L}, v)$ heavily depend on the choice of a valuation $v$, it is a fundamental problem to give concrete relations among Newton--Okounkov bodies associated with different kinds of valuations. 
In the case of flag and Schubert varieties, the following representation-theoretic polytopes can be realized as their Newton--Okounkov bodies:
\begin{enumerate}
\item[(i)] Berenstein--Littelmann--Zelevinsky's string polytopes \cite{Kav},
\item[(ii)] Nakashima--Zelevinsky polytopes \cite{FN},
\item[(iii)] FFLV (Feigin--Fourier--Littelmann--Vinberg) polytopes \cite{FeFL3, Kir},
\end{enumerate}
where the attached references are the ones giving realizations as Newton--Okounkov bodies. The set of lattice points in every polytope of {\rm (i)}--{\rm (iii)} gives a parametrization of a specific basis of an irreducible highest weight module of a semisimple Lie algebra. 
In particular, Berenstein--Littelmann--Zelevinsky's string polytopes and Nakashima--Zelevinsky polytopes give polyhedral parametrizations of the dual canonical basis and the crystal basis (see \cite{BZ, Lit, Nak, NZ}). 
Our aim in the present paper is to relate these polytopes by applying iterated combinatorial mutations. Analogous relations between flat families and combinatorial mutations were previously given in many papers \cite{Ilt, Por, EH} mainly in the context of complexity-$1$ $T$-varieties. 
In particular, Ilten \cite[Appendix A]{EH} related a wall-crossing phenomenon for Newton--Okounkov bodies of complexity-$1$ $T$-varieties with combinatorial mutations.

Combinatorial mutations for lattice polytopes were introduced by Akhtar--Coates--Galkin--Kasprzyk \cite{ACGK} in the context of mirror symmetry for Fano manifolds. 
The original motivation in \cite{ACGK} is to classify Fano manifolds using combinatorial mutations. 
A Laurent polynomial $f$ in $m$-variables is said to be a \emph{mirror partner} of an $m$-dimensional Fano manifold $X$ if the period $\pi_f$ of $f$ coincides with the quantum period $\widehat{G}_X$ of $X$ (see \cite{ACGK, CCGGK} and references therein for more details). 
Two distinct Laurent polynomials can have the same period, and there can be many choices of mirror partners for the same Fano manifold $X$.
Combinatorial mutations for Laurent polynomials are useful to find such mirror partners since they preserve the period (see \cite[Lemma 1]{ACGK}).
The notion of combinatorial mutations for lattice polytopes just rephrases that for Laurent polynomials in terms of their Newton polytopes.
Hence combinatorial mutations for lattice polytopes are also expected to preserve some combinatorial or mirror-symmetric properties, which is one motivation of this research. 
For instance, they preserve the Ehrhart series of the dual polytope.

\begin{ntn}
We adopt the standard notation in toric geometry. Let $N \simeq \z^m$ be a $\z$-lattice of rank $m$, and $M \coloneqq {\rm Hom}_\z(N, \z) \simeq \z^m$ its dual lattice. 
We write $N_\r \coloneqq N \otimes_\z \r$ and $M_\r \coloneqq M \otimes_\z \r$. 
Denote by $\langle \cdot, \cdot \rangle \colon M_\r \times N_\r \rightarrow \r$ the canonical pairing. 
\end{ntn}

The combinatorial mutation in \cite{ACGK} is an operation for lattice polytopes in $N_\r$. For a rational convex polytope $\Delta \subseteq M_\r$ with a unique interior lattice point ${\bm a}$, we define its dual $\Delta^\vee$ to be the polar dual of its translation: 
\begin{align*}
\Delta^\vee &\coloneqq (\Delta - {\bm a})^\ast\\
&= \{v \in N_\r \mid \langle u -{\bm a}, v \rangle \geq -1\ {\rm for\ all}\ u \in \Delta\}. 
\end{align*}
A combinatorial mutation on $\Delta^\vee$ corresponds to a piecewise-linear operation on $\Delta - {\bm a}$, which is extended to the whole of $M_\r$. 
We call such an operation a \emph{combinatorial mutation in} $M_\r$. We consider this framework when $\Delta$ is a Newton--Okounkov body of a flag variety.

To state our results explicitly, let $G$ be a simply-connected semisimple algebraic group over $\mathbb{C}$, $B$ a Borel subgroup of $G$, and $W$ the Weyl group. 
We denote by $X(w) \subseteq G/B$ the Schubert variety corresponding to $w \in W$, by $R(w)$ the set of reduced words for $w \in W$, and by $w_0 \in W$ the longest element. 
Note that the Schubert variety $X(w_0)$ corresponding to $w_0$ coincides with $G/B$. 
Let $P_+$ be the set of dominant integral weights, $\mathcal{L}_\lambda$ the globally generated line bundle on $X(w)$ associated with $\lambda \in P_+$, and $\rho \in P_+$ the half sum of the positive roots. 
For ${\bm i} \in R(w)$ and $\lambda \in P_+$, we denote by $\Delta_{\bm i} (\lambda)$ (resp., $\widetilde{\Delta}_{\bm i} (\lambda)$) the corresponding string polytope (resp., the corresponding Nakashima--Zelevinsky polytope). 
In order to relate string polytopes and Nakashima--Zelevinsky polytopes by combinatorial mutations, we use the theory of cluster algebras. 
The theory of cluster algebras was originally introduced by Fomin--Zelevinsky \cite{FZ:ClusterI, FZ:ClusterIV} to develop a combinatorial approach to total positivity in reductive groups and to the dual canonical basis. 
Fock--Goncharov \cite{FG} introduced a pair $(\mathcal{A}, \mathcal{X})$ of cluster varieties, called a cluster ensemble, which is a geometric counterpart of the theory of cluster algebras. 
Gross--Hacking--Keel--Kontsevich \cite{GHKK} developed the theory of cluster ensembles using methods in mirror symmetry, and proved that the theory of cluster algebras also can be used to construct toric degenerations of projective varieties. 
We denote by $U_w ^- \subseteq G$ the unipotent cell associated with $w \in W$, which is naturally thought of as an open subvariety of $X(w)$. 
Berenstein--Fomin--Zelevinsky \cite{BFZ} proved that the coordinate ring $\c[U_w ^-]$ admits an upper cluster algebra structure. 
When $G$ is simply-laced, the first named author and Oya \cite{FO2} constructed a family $\{\Delta(X(w), \mathcal{L}_\lambda, v_{\mathbf s})\}_{{\mathbf s} \in \mathcal{S}}$ of Newton--Okounkov bodies parametrized by the set of seeds for $\c[U_w ^-]$ such that 
\begin{itemize}
\item this family contains $\Delta_{\bm i} (\lambda)$ and $\widetilde{\Delta}_{\bm i} (\lambda)$ for all ${\bm i} \in R(w)$ up to unimodular transformations,
\item the Newton--Okounkov bodies $\Delta(X(w), \mathcal{L}_\lambda, v_{\mathbf s})$, ${\mathbf s} \in \mathcal{S}$, are all rational convex polytopes, 
\item the Newton--Okounkov bodies $\Delta(X(w), \mathcal{L}_\lambda, v_{\mathbf s})$, ${\mathbf s} \in \mathcal{S}$, are all related by tropicalized cluster mutations,
\end{itemize}
where a \emph{unimodular transformation} is a linear transformation which is represented by a unimodular matrix.
If $w = w_0$ and $\lambda = 2\rho$, then the Newton--Okounkov body $\Delta(G/B, \mathcal{L}_{2\rho}, v_{\mathbf s})$ contains exactly one lattice point ${\bm a}_{\mathbf s}$ in its interior, and the dual $\Delta(G/B, \mathcal{L}_{2 \rho}, v_{\mathbf s})^\vee$ is a lattice polytope (see \cref{t:unique_lattice_point_Steinert} and \cref{c:unique_lattice_point_reflexivity}). Realizing tropicalized cluster mutations as combinatorial mutations in $M_\r$, we obtain the following.

\begin{theor1}[{\cref{t:combinatorial_mutations_NO}}]
If $G$ is simply-laced, then the following hold.
\begin{enumerate}
\item[{\rm (1)}] For fixed $w \in W$ and $\lambda \in P_+$, the Newton--Okounkov bodies $\Delta(X(w), \mathcal{L}_\lambda, v_{\mathbf s})$, ${\mathbf s} \in \mathcal{S}$, are all related by combinatorial mutations in $M_\r$ up to unimodular transformations.
\item[{\rm (2)}] For $w = w_0$ and $\lambda = 2\rho$, the translated polytopes $\Delta(G/B, \mathcal{L}_{2\rho}, v_{\mathbf s}) -{\bm a}_{\mathbf s}$, ${\mathbf s} \in \mathcal{S}$, are all related by combinatorial mutations in $M_\r$ up to unimodular transformations. 
As a consequence, the dual polytopes $\Delta(G/B, \mathcal{L}_{2\rho}, v_{\mathbf s})^\vee$, ${\mathbf s} \in \mathcal{S}$, are all related by combinatorial mutations in $N_\r$ up to unimodular transformations.
\end{enumerate}
\end{theor1}

In order to relate FFLV polytopes with these Newton--Okounkov bodies, we use Ardila--Bliem--Salazar's transfer map \cite{ABS} between the Gelfand--Tsetlin polytope $GT(\lambda)$ and the FFLV polytope $FFLV(\lambda)$ in type $A_n$, where $\lambda \in P_+$. 
A \emph{unimodular affine transformation} is a composition of a unimodular transformation and a translation by an integer vector. 
We say that two polytopes are \emph{unimodularly equivalent} if they are related by a unimodular affine transformation.
Littelmann \cite{Lit} showed that the string polytope $\Delta_{\bm i} (\lambda)$ associated with specific ${\bm i} \in R(w_0)$ is unimodularly equivalent to the Gelfand--Tsetlin polytope $GT(\lambda)$; see \cref{ex:GT_polytopes_type_A}.
We realize their transfer map as a composition of combinatorial mutations in $M_\r$ (see Theorem \ref{t:relation_with_FFLV_type_A}). 
Combining this with Theorem 1, we obtain the following in type $A_n$.

\begin{theor2}[{see Theorems \ref{t:combinatorial_mutations_NO}, \ref{t:relation_with_FFLV_type_A}}]
If $G = SL_{n+1} (\c)$, then the following hold.
\begin{enumerate}
\item[{\rm (1)}] For fixed $\lambda \in P_+$, the string polytopes $\Delta_{\bm i} (\lambda)$, ${\bm i} \in R(w_0)$, the Nakashima--Zelevinsky polytopes $\widetilde{\Delta}_{\bm i} (\lambda)$, ${\bm i} \in R(w_0)$, and the FFLV polytope $FFLV(\lambda)$ are all related by combinatorial mutations in $M_\r$ up to unimodular equivalence.
\item[{\rm (2)}] The polytopes in 
\[\{\Delta_{\bm i} (2\rho)^\vee \mid {\bm i} \in R(w_0)\} \cup \{\widetilde{\Delta}_{\bm i} (2\rho)^\vee \mid {\bm i} \in R(w_0)\} \cup \{FFLV(2\rho)^\vee\}\] 
are all related by combinatorial mutations in $N_\r$ up to unimodular transformations.
\end{enumerate}
\end{theor2}

Note that Ardila--Bliem--Salazar \cite{ABS} gave such a transfer map also in type $C_n$. Since their transfer map in type $C_n$ can be also described as a composition of combinatorial mutations in $M_\r$, we obtain the following.

\begin{theor3}[{\cref{t:relation_with_FFLV_type_C}}]
If $G = Sp_{2n} (\c)$, then the following hold.
\begin{enumerate}
\item[{\rm (1)}] For fixed $\lambda \in P_+$, the Gelfand--Tsetlin polytope $GT_{C_n}(\lambda)$ and the FFLV polytope $FFLV_{C_n}(\lambda)$ are related by combinatorial mutations in $M_\r$ up to translations by integer vectors.
\item[{\rm (2)}] The dual polytopes $GT_{C_n}(2\rho)^\vee$ and $FFLV_{C_n}(2\rho)^\vee$ are related by combinatorial mutations in $N_\r$.
\end{enumerate}
\end{theor3}

\begin{ack}\normalfont
The authors are grateful to Xin Fang and Ghislain Fourier for explaining relations with polytopes introduced in \cite{FF, FFLP}. 
The authors would like to thank Alexander Kasprzyk for informing them of some related works.
The authors would also like to thank the anonymous referees for reading the manuscript carefully and for suggesting many improvements.
\end{ack}

\bigskip

\section{Newton--Okounkov bodies arising from cluster structures}\label{s:NO_body}

In order to relate string polytopes and Nakashima--Zelevinsky polytopes by combinatorial mutations, we use Newton--Okounkov bodies of flag varieties arising from cluster structures. In Section \ref{ss:NO_body}, we recall the definitions of higher rank valuations and Newton--Okounkov bodies. We also review their basic properties. In Section \ref{ss:cluster_algebra}, we define valuations using cluster structures, following \cite{FO2}.

\subsection{Basic definitions on Newton--Okounkov bodies}\label{ss:NO_body}

We first recall the definition of Newton--Okounkov bodies, following \cite{HK, Kav, KK1, KK2}. Let $R$ be a $\mathbb{C}$-algebra without nonzero zero-divisors, and $m \in \z_{>0}$. We fix a total order $\leq$ on $\mathbb{Z}^m$ respecting the addition. 

\begin{defi}\label{defval}\normalfont
	A map $v \colon R \setminus \{0\} \rightarrow \mathbb{Z}^m$ is called a \emph{valuation} on $R$ with values in $\z^m$ if for each $\sigma, \tau \in R \setminus \{0\}$ and $c \in \mathbb{C}^\times \coloneqq \mathbb{C} \setminus \{0\}$, we have
	\begin{enumerate}
		\item[{\rm (i)}] $v(\sigma \cdot \tau) = v(\sigma) + v(\tau)$,
		\item[{\rm (ii)}] $v(c \cdot \sigma) = v(\sigma)$, 
		\item[{\rm (iii)}] $v (\sigma + \tau) \geq {\rm min} \{v(\sigma), v(\tau) \}$ unless $\sigma + \tau = 0$. 
	\end{enumerate}
\end{defi}

For ${\bm a} \in \z^m$ and a valuation $v$ on $R$ with values in $\z^m$, we define a $\c$-subspace $R_{\bm a} \subseteq R$ as follows:
\[
R_{\bm a} \coloneqq \{\sigma \in R \setminus \{0\} \mid v(\sigma) \geq {\bm a}\} \cup \{0\}.
\]
Then the \emph{leaf} $\widehat{R}_{\bm a}$ above ${\bm a} \in \z^m$ is defined by $\widehat{R}_{\bm a} \coloneqq R_{\bm a}/\bigcup_{{\bm a} < {\bm b}} R_{\bm b}$. We say that a valuation $v$ has \emph{$1$-dimensional leaves} if $\dim_\c(\widehat{R}_{\bm a}) = 0\ {\rm or}\ 1$ for all ${\bm a} \in \z^m$. 

\begin{ex}\label{ex:lowest_term_valuation}
Fix a total order $\leq$ on $\mathbb{Z}^m$ respecting the addition, and let $\mathbb{C}(z_1, \ldots, z_m)$ be the field of rational functions in $m$ variables. The total order $\leq$ on $\mathbb{Z}^m$ gives a total order (denoted by the same symbol $\leq$) on the set of Laurent monomials in $z_1, \ldots, z_m$ by
\begin{center}
$z_1 ^{a_1} \cdots z_m ^{a_m} \leq z_1 ^{a_1 ^\prime} \cdots z_m ^{a_m ^\prime}$ if and only if $(a_1, \ldots, a_m) \leq (a_1 ^\prime, \ldots, a_m ^\prime)$. 
\end{center}
We define a map $v^{\rm low}_{\leq} \colon \mathbb{C}(z_1, \ldots, z_m) \setminus \{0\} \rightarrow \mathbb{Z}^m$ by
\begin{itemize}
\item $v^{\rm low} _{\leq} (f) \coloneqq (a_1, \ldots, a_m)$ for
\[
f = c z_1 ^{a_1} \cdots z_m ^{a_m} + ({\rm higher\ terms}) \in \mathbb{C}[z_1 ^{\pm 1}, \ldots, z_m ^{\pm 1}] \setminus \{0\},
\]
where $c \in \mathbb{C}^\times$, and we mean by ``(higher terms)'' a linear combination of Laurent monomials bigger than $z_1 ^{a_1} \cdots z_m ^{a_m}$ with respect to the total order $\leq$,
\item $v^{\rm low}_{\leq} (f/g) \coloneqq v^{\rm low} _{\leq} (f) - v^{\rm low} _{\leq} (g)$ for $f, g \in \mathbb{C}[z_1 ^{\pm 1}, \ldots, z_m ^{\pm 1}] \setminus \{0\}$.
\end{itemize}
Then this map $v^{\rm low} _{\leq}$ is a valuation with respect to the total order $\leq$, which has $1$-dimensional leaves. This is called the {\it lowest term valuation} with respect to $\leq$.
\end{ex}

\begin{defi}[{see \cite[Definition 1.10]{KK2}}]\normalfont\label{Newton--Okounkov body}
Let $X$ be an irreducible normal projective variety over $\c$, $\mathcal{L}$ a line bundle on $X$ generated by global sections, and $m \coloneqq \dim_\c (X)$. 
Take a valuation $v \colon \mathbb{C}(X) \setminus \{0\} \rightarrow \z^m$ which has $1$-dimensional leaves, and fix a nonzero section $\tau \in H^0 (X, \mathcal{L})$. 
Let us define a semigroup $S(X, \mathcal{L}, v, \tau) \subseteq \z_{>0} \times \z^m$ by \[S(X, \mathcal{L}, v, \tau) \coloneqq \bigcup_{k \in \z_{>0}} \{(k, v(\sigma / \tau^k)) \mid \sigma \in H^0(X, \mathcal{L}^{\otimes k}) \setminus \{0\}\},\] and denote by $C(X, \mathcal{L}, v, \tau) \subseteq \r_{\geq 0} \times \r^m$ the smallest real closed cone containing $S(X, \mathcal{L}, v, \tau)$. 
We define a convex set $\Delta(X, \mathcal{L}, v, \tau) \subseteq \r^m$ by \[\Delta(X, \mathcal{L}, v, \tau) \coloneqq \{{\bm a} \in \r^m \mid (1, {\bm a}) \in C(X, \mathcal{L}, v, \tau)\},\] which is called the {\it Newton--Okounkov body} of $(X, \mathcal{L})$ associated with $(v, \tau)$.
In the notation of \cite[Definition 1.10]{KK2}, our Newton--Okounkov body $\Delta(X, \mathcal{L}, v, \tau)$ is $\Delta(S, M)$ for $S = S(X, \mathcal{L}, v, \tau)$ and $M = L(S) \cap (\r_{\geq 0} \times \r^m)$, where $L(S) \subseteq \r \times \r^m$ denotes the linear span of $S$ (see also \cite[Section 3.2]{KK2} and \cite[Section 1.2]{Kav}).
\end{defi}

It follows by \cite[Theorem 2.30]{KK2} that the Newton--Okounkov body $\Delta(X, \mathcal{L}, v, \tau)$ is a convex body, i.e., a compact convex set. 
If $\mathcal{L}$ is ample, then we see by \cite[Corollary 3.2]{KK2} that $\dim_\r \Delta(X, \mathcal{L}, v, \tau) = m = \dim_\c (X)$.
When $\mathcal{L}$ is not ample, the real dimension of $\Delta(X, \mathcal{L}, v, \tau)$ can be smaller than $m$.
By definition, we have 
\[0 = v(\tau/\tau) \in \Delta(X, \mathcal{L}, v, \tau).\]
Since $S(X, \mathcal{L}, v, \tau)$ is a semigroup, the definition of Newton--Okounkov bodies implies that 
\[
\Delta(X, \mathcal{L}^{\otimes k}, v, \tau^k) = k \Delta(X, \mathcal{L}, v, \tau)
\]
for all $k \in \z_{> 0}$. 

\begin{rem}\label{independence}
For another nonzero section $\tau^\prime \in H^0 (X, \mathcal{L})$, it follows that 
\[
S(X, \mathcal{L}, v, \tau^\prime) \cap (\{k\} \times \mathbb{Z}^m) = (S(X, \mathcal{L}, v, \tau) \cap (\{k\} \times \mathbb{Z}^m)) + (0, k v(\tau/\tau^\prime))
\] 
for all $k \in \mathbb{Z}_{>0}$. From this, we have
\[
\Delta(X, \mathcal{L}, v, \tau^\prime) = \Delta(X, \mathcal{L}, v, \tau) + v(\tau/\tau^\prime),
\] 
which implies that $\Delta(X, \mathcal{L}, v, \tau)$ does not essentially depend on the choice of $\tau$. Hence we also denote it simply by $\Delta(X, \mathcal{L}, v)$.
\end{rem} 

\subsection{Cluster algebras and valuations}\label{ss:cluster_algebra}

The first named author and Oya \cite{FO2} constructed valuations using the theory of cluster algebras. 
In this subsection, we review this construction. 
We first recall the definition of (upper) cluster algebras of geometric type, following \cite{BFZ, FZ:ClusterIV}. 
Note that we use the notation in \cite{FG, GHKK}. 
Fix a finite set $J$ and a subset $J_{\rm uf} \subseteq J$. 
We write $J_{\rm fr} \coloneqq J \setminus J_{\rm uf}$. Let $\mathcal{F} \coloneqq \mathbb{C}(z_j \mid j \in J)$ be the field of rational functions in $|J|$ variables. 
For a $J$-tuple $\mathbf{A} = (A_j)_{j \in J}$ of elements of $\mathcal{F}$ and $\varepsilon = (\varepsilon_{i, j})_{i \in J_{\rm uf}, j \in J} \in {\rm Mat}_{J_{\rm uf} \times J}(\mathbb{Z})$, the pair ${\mathbf s} = (\mathbf{A}, \varepsilon)$ is called a \emph{seed} of $\mathcal{F}$ if 
\begin{itemize}
	\item[(i)] $\mathbf{A}$ forms a free generating set of the field $\mathcal{F}$, and
   \item[(ii)] the $J_{\rm uf} \times J_{\rm uf}$-submatrix $\varepsilon^{\circ}$ of $\varepsilon$ is skew-symmetrizable, that is, there is $(d_i)_{i \in J_{\rm uf}} \in \mathbb{Z}^{J_{\rm uf}}_{>0}$ such that $d_i \varepsilon_{i, j} = -d_j \varepsilon_{j, i}$ for all $i, j \in J_{\rm uf}$.  
\end{itemize}
In this case, we call $\varepsilon$ the \emph{exchange matrix} of ${\mathbf s}$. Note that the exchange matrix $\varepsilon$ is transposed to the one in \cite[Section 2]{FZ:ClusterIV}.

Let ${\mathbf s} = (\mathbf{A}, \varepsilon) = ((A_j)_{j \in J}, (\varepsilon_{i, j})_{i \in J_{\rm uf}, j \in J})$ be a seed of $\mathcal{F}$. We write $[c]_+ \coloneqq \max\{c, 0\}$ for $c \in \r$. For $k \in J_{\rm uf}$, define the \emph{mutation} $\mu_k ({\mathbf s}) = (\mu_k (\mathbf{A}), \mu_k (\varepsilon)) = ((A^\prime _j)_{j \in J}, (\varepsilon^\prime _{i, j})_{i \in J_{\rm uf}, j \in J})$ in direction $k$ as follows:
\begin{align*}
\varepsilon^\prime _{i, j} \coloneqq \begin{cases}
-\varepsilon_{i, j}&\ \text{if}\ i=k\ {\rm or}\ j=k,\\
\varepsilon_{i, j} + {\rm sgn}(\varepsilon_{i, k})[\varepsilon_{i, k} \varepsilon_{k, j}]_+&\ {\rm otherwise},
\end{cases}
\end{align*}
\begin{align*}
A_j ^\prime \coloneqq \begin{cases}
\displaystyle \frac{\prod_{\ell \in J} A_\ell ^{[\varepsilon_{k, \ell}]_+} + \prod_{\ell \in J} A_\ell ^{[-\varepsilon_{k, \ell}]_+}}{A_k} &\ {\rm if}\ j = k,\\
A_j &\ {\rm otherwise}
\end{cases}
\end{align*}
for $i \in J_{\rm uf}$ and $j \in J$. 
Then $\mu_k ({\mathbf s})$ is also a seed of $\mathcal{F}$, and it follows that $\mu_k \mu_k ({\mathbf s}) = {\mathbf s}$. 
We say that two seeds ${\mathbf s}$ and ${\mathbf s}^\prime$ are \emph{mutation equivalent} if there exists a sequence $(k_1, k_2, \ldots, k_j)$ in $J_{\rm uf}$ such that 
\[{\mathbf s}^\prime = \mu_{k_j} \cdots \mu_{k_2} \mu_{k_1} ({\mathbf s}).\]

\begin{ex}\label{ex:basic_example_cluster_mutation}
Let $J = \{1, \ldots, 6\}$, and $J_{\rm uf} = \{1, 2, 3\}$. 
Take a seed ${\mathbf s} = (\mathbf{A}, \varepsilon)$ of $\mathcal{F}$ whose exchange matrix $\varepsilon$ is given by 
\[\varepsilon = \begin{pmatrix}
0 & -1 & 1 & 0 & 0 & 0 \\
1 & 0 & -1 & -1 & 1 & 0 \\
-1 & 1 & 0 & 0 & -1 & 1 
\end{pmatrix}.\]
Then the mutation $\mu_2 ({\mathbf s}) = (\mu_2 (\mathbf{A}), \mu_2 (\varepsilon))$ in direction $2$ is given as follows: 
\begin{align*}
&\mu_2 (\varepsilon) = \begin{pmatrix}
0 & 1 & 0 & -1 & 0 & 0 \\
-1 & 0 & 1 & 1 & -1 & 0 \\
0 & -1 & 0 & 0 & 0 & 1 
\end{pmatrix},\\
&A_j ^\prime = \begin{cases}
\displaystyle \frac{A_1 A_5 + A_3 A_4}{A_2} &\ {\rm if}\ j = 2,\\
A_j &\ {\rm otherwise}
\end{cases}
\end{align*}
for $j \in J$, where we write $\mathbf{A} = (A_j)_{j \in J}$ and $\mu_2 (\mathbf{A}) = (A^\prime _j)_{j \in J}$. 
\end{ex}

Let $\mathbb{T}$ be the $|J_{\rm uf}|$-regular tree whose edges are labeled by $J_{\rm uf}$ such that the $|J_{\rm uf}|$-edges emanating from each vertex have different labels. 
Let us write $t \overset{k}{\text{---}} t^\prime$ when $t, t^\prime \in \mathbb{T}$ are joined by an edge labeled by $k \in J_{\rm uf}$. 
Note that $\mathbb{T}$ is an infinite tree when $|J_{\rm uf}| \geq 2$.

\begin{ex}
If $J_{\rm uf} = \{1, 2\}$, then the $2$-regular tree $\mathbb{T}$ is given as 
\begin{align*}
\begin{xy}
\ar@{-}^-{1} (50,0) *\cir<3pt>{}="A";
(60,0) *\cir<3pt>{}="B"
\ar@{-}^-{2} "B";(70,0) *\cir<3pt>{}="C"
\ar@{-}^-{1} "C";(80,0) *\cir<3pt>{}="D"
\ar@{-}^-{2} "D";(90,0) *\cir<3pt>{}="E"
\ar@{-} "E";(95,0)
\ar@{-} "A";(45,0)
\ar@{.} (95,0);(100,0)^*!U{}
\ar@{.} (45,0);(40,0)^*!U{}
\end{xy}
\end{align*}
\end{ex}

An assignment $\mathcal{S} = \{{\mathbf s}_t\}_{t \in \mathbb{T}} = \{(\mathbf{A}_t, \varepsilon_t)\}_{t \in \mathbb{T}}$ of a seed ${\mathbf s}_t = (\mathbf{A}_t, \varepsilon_t)$ of $\mathcal{F}$ to each vertex $t \in \mathbb{T}$ is called a \emph{cluster pattern} if $\mu_k ({\mathbf s}_t) = {\mathbf s}_{t^\prime}$ whenever $t \overset{k}{\text{---}} t^\prime$. 
Given a cluster pattern $\mathcal{S} = \{{\mathbf s}_t = (\mathbf{A}_t, \varepsilon_t)\}_{t \in \mathbb{T}}$, we write 
\[
\mathbf{A}_t = (A_{j; t})_{j \in J}, \quad \varepsilon_t = (\varepsilon_{i, j} ^{(t)})_{i \in J_{\rm uf}, j \in J}.
\]

\begin{defi}[{see \cite[Definitions 1.6 and 1.11]{BFZ}}]\normalfont
For a cluster pattern $\mathcal{S}$, the \emph{upper cluster algebra} $\mathscr{U}(\mathcal{S})$ \emph{of geometric type} is defined by 
\[
\mathscr{U}(\mathcal{S}) \coloneqq \bigcap_{t \in \mathbb{T}} \c[A_{j; t} ^{\pm 1} \mid j \in J] \subseteq \mathcal{F}.
\]
\end{defi}

Usually, we fix $t_0 \in \mathbb{T}$ and construct a cluster pattern $\mathcal{S} = \{{\mathbf s}_t\}_{t \in \mathbb{T}}$ from one seed ${\mathbf s}_{t_0}$ of $\mathcal{F}$. 
In this case, we call ${\mathbf s}_{t_0}$ the \emph{initial seed}. 

\begin{thm}[{\cite[Theorem 3.1]{FZ:ClusterI}}]\label{t:laurentpheno}
Let $\mathcal{S}$ be a cluster pattern. 
Then the set $\{A_{j; t} \mid t \in \mathbb{T},\ j \in J\}$ is included in the upper cluster algebra $\mathscr{U}(\mathcal{S})$; this property is called the \emph{Laurent phenomenon}.
\end{thm}

In the rest of this subsection, assume that 
\begin{enumerate}
\item[($\diamondsuit$)] the exchange matrix $\varepsilon_{t_0}$ of ${\mathbf s}_{t_0}$ is of full rank for some $t_0 \in \mathbb{T}$,
\end{enumerate}
which implies that the exchange matrix $\varepsilon_t$ of ${\mathbf s}_t$ is of full rank for all $t \in \mathbb{T}$ (see \cite[Lemma 3.2]{BFZ}). 

\begin{defi}[{\cite[Definition 3.1.1]{Qin}}]\label{d:order}\normalfont
Let $\mathcal{S} = \{{\mathbf s}_t = (\mathbf{A}_t, \varepsilon_t)\}_{t \in \mathbb{T}}$ be a cluster pattern, and fix $t \in \mathbb{T}$. Then we define a partial order $\preceq_{\varepsilon_t}$ on $\mathbb{Z}^{J}$ as follows: for ${\bm a}, {\bm a}^\prime \in \mathbb{Z}^{J}$,
\[
{\bm a} \preceq_{\varepsilon_t} {\bm a}^\prime\ {\rm if\ and\ only\ if}\ {\bm a} = {\bm a}^\prime + {\bm u} \varepsilon_t\ {\rm for\ some}\ {\bm u} \in \mathbb{Z}_{\geq 0}^{J_{\rm uf}},
\]
where elements of $\mathbb{Z}^{J}$ (resp., $\mathbb{Z}_{\geq 0}^{J_{\rm uf}}$) are regarded as row vectors. This $\preceq_{\varepsilon_t}$ is called the \emph{dominance order} associated with $\varepsilon_t$.
\end{defi}

\begin{defi}[{\cite[Definition 3.8]{FO2}}]\normalfont\label{d:main_valuation}
Let $\mathcal{S} = \{{\mathbf s}_t = (\mathbf{A}_t, \varepsilon_t)\}_{t \in \mathbb{T}}$ be a cluster pattern, and fix $t \in \mathbb{T}$. We take a total order $\leq_t$ on $\z^J$ which refines the opposite order $\preceq_{\varepsilon_t} ^{\rm op}$ of the dominance order $\preceq_{\varepsilon_t}$. 
The total order $\leq_t$ on $\z^J$ gives a total order (denoted by the same symbol $\leq_t$) on the set of Laurent monomials in $A_{j;t}$, $j \in J$, as follows: 
\begin{align*}
\prod_{j \in J} A_{j;t} ^{a_j} \leq_t \prod_{j \in J} A_{j;t} ^{a_j ^\prime}\quad \text{if and only if}\quad (a_j)_{j \in J} \leq_t (a_j ^\prime)_{j \in J}. 
\end{align*}
Then we denote by $v_{{\mathbf s}_t}$ (or simply by $v_t$) the lowest term valuation $v^{\rm low} _{\leq_t}$ on $\mathcal{F} = \c(A_{j; t} \mid j \in J)$ with respect to $\leq_t$. 
\end{defi} 

Following \cite[Conjecture 7.12]{FZ:ClusterIV} and \cite[Section 4]{FG}, we define a \emph{tropicalized cluster mutation} as follows: for $t \overset{k}{\text{---}} t^\prime$,
\[\mu_k ^T \colon \r^J \rightarrow \r^J,\ (g_j)_{j \in J} \mapsto (g' _j)_{j \in J},\] 
where 
\begin{align*}
g' _j \coloneqq
\begin{cases}
g_j + [-\varepsilon_{k, j} ^{(t)}]_+ g_k + \varepsilon_{k, j} ^{(t)} [g_k]_+ &(j \neq k),\\
-g_j &(j = k)
\end{cases}
\end{align*}
for $j \in J$; this is the tropicalization of the mutation $\mu_k$ for the Fock--Goncharov dual $\mathcal{A}^\vee$ of the $\mathcal{A}$-cluster variety (see also \cite[Section 2 and Definition A.4]{GHKK}).
As we review in Section \ref{ss:Schubert_as_cluster}, the tropicalized cluster mutation $\mu_k ^T$ can be used to connect Newton--Okounkov bodies associated with $v_t$ and $v_{t^\prime}$.

\begin{ex}
Let $J = \{1, \ldots, 6\}$, and $J_{\rm uf} = \{1, 2, 3\}$. 
If $\varepsilon_t$ is given as the matrix $\varepsilon$ in \cref{ex:basic_example_cluster_mutation} for some $t \in \mathbb{T}$, then the tropicalized cluster mutation $\mu_2 ^T \colon \r^6 \rightarrow \r^6$ at $t$ is given by
\[\mu_2 ^T (g_1, \ldots, g_6) = (g_1 + [g_2]_+, -g_2, g_3 - [-g_2]_+, g_4 - [-g_2]_+, g_5 + [g_2]_+, g_6).\]
\end{ex}

\bigskip

\section{Combinatorial mutations and tropicalized cluster mutations}

In this section, we recall the notion of combinatorial mutations for lattice polytopes which was developed by Akhtar--Coates--Galkin--Kasprzyk in \cite{ACGK}. 
Then we realize tropicalized cluster mutations as combinatorial mutations.
There are two kinds of combinatorial mutations: one is the operation in $N_\r$-side and the other one is in $M_\r$-side. 
Our main interest is the operation in $M_\r$-side (see Definition~\ref{def:M_side}) and this is originally defined as a ``dual version'' of the operation in $N_\r$-side. See Proposition~\ref{prop:compatibility}. 

\subsection{Basic definitions on combinatorial mutations}

We first introduce combinatorial mutations for lattice polytopes in $N_\r$. 
Let $P \subseteq N_\r$ be a lattice polytope, and take $w \in M$. For $h \in \z$, write
$$H_{w,h} \coloneqq \{v \in N_\r \mid \langle w,v \rangle = h\}, \; \text{ and }\; P_{w,h} \coloneqq P \cap H_{w,h}.$$
We use the notation $w^\perp$ instead of $H_{w,0}$. Let $V(P) \subseteq N$ denote the set of vertices of $P$. 
For each subset $A \subseteq N_\r$, we set $A+\emptyset = \emptyset +A = \emptyset$. 

\begin{defi}[{\cite[Definition 5]{ACGK}}]\label{def:N_side}
Let $P \subseteq N_\r$ be a lattice polytope, $w \in M$ a primitive vector, and $F$ a lattice polytope which sits in $w^\perp$.
We say that the \textit{combinatorial mutation} ${\rm mut}_w(P,F)$ of $P$ in $N_\r$ is \textit{well-defined} if for every negative integer $h$, there exists a possibly-empty lattice polytope $G_h \subseteq N_\r$ such that the inclusions
\begin{align}\label{eq:well-defined}
V(P) \cap H_{w,h} \subseteq G_h + |h|F \subseteq P_{w,h}
\end{align}
hold. 
In this case, the combinatorial mutation ${\rm mut}_w(P,F)$ of $P$ is a lattice polytope defined as follows:
\[{\rm mut}_w(P,F) \coloneqq {\rm conv}\left( \bigcup_{h \leq -1}G_h \cup \bigcup_{h \geq 0}(P_{w,h} + hF) \right) \subseteq N_\r.\] 
Note that $G_h$ and $P_{w,h}+hF$ are empty except for finitely many $h$'s. 
\end{defi}

It is proved in \cite[Proposition 1]{ACGK} that ${\rm mut}_w(P,F)$ is independent of the choice of $\{G_h\}_h$. 

\begin{rem}
In \cite{Hig}, the definition of combinatorial mutations in $N_\r$ has been extended to rational convex polytopes and unbounded polyhedra. See \cite[Section 2]{Hig} for more details. 
\end{rem}

Next, we introduce another operation, which is a piecewise-linear transformation on $M_\r$. 

\begin{defi}[{\cite[Section 3]{ACGK}; see also \cite[Definition 3.1]{Hig}}]\label{def:M_side}
Let $w \in M$ be a primitive vector, and take a lattice polytope $F$ which sits in $w^\perp$. We define a map $\varphi_{w,F} \colon M_\r \rightarrow M_\r$ by 
\[\varphi_{w,F}(u) \coloneqq u-u_{\rm min}w,\]
where $u_{\rm min} \coloneqq \min\{\langle u, v \rangle \mid v \in F\}$. 
We call the piecewise-linear map $\varphi_{w,F}$ a \textit{combinatorial mutation in $M_\r$}. 
\end{defi}

Indeed, the combinatorial mutation $\varphi_{w,F}$ in $M_\r$ is compatible with the one in $N_\r$ through the polar dual. More precisely, we see the following. 

\begin{prop}[{\cite{ACGK} and \cite[Proposition 3.2]{Hig}}]\label{prop:compatibility}
Let $P \subseteq N_\r$ be a lattice polytope containing the origin. Take a primitive vector $w \in M$, and fix a lattice polytope $F \subseteq w^\perp$. Assume that ${\rm mut}_w(P,F)$ is well-defined.
Then it holds that \[\varphi_{w,F}(P^\ast) = {\rm mut}_w(P,F)^\ast.\] 
\end{prop}

\begin{prop}[{\cite[Proposition 3.4]{Hig}}]\label{p:compatibility_well-defined}
Fix a primitive vector $w \in M$ and a lattice polytope $F \subseteq w^\perp$. 
Let $Q \subseteq M_\r$ be a rational convex polytope containing the origin. 
Then ${\rm mut}_w(Q^*,F)$ is well-defined if and only if $\varphi_{w,F}(Q)$ is convex. 
\end{prop}

\begin{ex}
Consider the lattice polygon 
\[
P={\rm conv}((1,1),(0,1),(-1,-1),(0,-1)) \subseteq N_\r \cong \r^2.
\] 
Let $w=(0,-1) \in M$, and $F={\rm conv}((0,0),(1,0)) \subseteq w^\perp$. 

By setting $G_{-1}=\{(0,1)\}$, we see that $V(P) \cap H_{w,-1} \subseteq G_{-1} + F = P_{w,-1}$. Hence ${\rm mut}_w(P,F)$ is well-defined and 
$${\rm mut}_w(P,F)={\rm conv}(G_{-1} \cup (P_{w,1} +F))={\rm conv}((0,1),(-1,-1),(1,-1)) \subseteq N_\r.$$ 
By taking the polar dual of this polytope, we obtain that $${\rm mut}_w(P,F)^\ast = {\rm conv}((0,1),(-2,-1),(2,-1)) \subseteq M_\r.$$ 

On the other hand, it holds that 
\[
P^\ast = {\rm conv}((0,-1),(2,-1),(0,1),(-2,1)) \subseteq M_\r.\] 
Now, we apply $\varphi_{w,F}$ to $P^\ast$. By definition, we have
\begin{align*}
\varphi_{w,F}((x,y))&=(x,y)-\min\{\langle (x,y), (0,0) \rangle, \langle (x,y), (1,0) \rangle\}(0,-1) \\ 
&= (x,y) - \min\{0,x\}(0,-1) \\
&=\begin{cases}
(x,y) \ &\text{ if }x \geq 0, \\
(x,x+y) &\text{ if }x \leq 0,
\end{cases}
\end{align*}
which implies that
\begin{align*}
\varphi_{w,F}(P^\ast \cap \{(x,y) \in \r^2 \mid x \geq 0\}) &= {\rm conv}((0,-1),(2,-1),(0,1)),\text{ and }\\
\varphi_{w,F}(P^\ast \cap \{(x,y) \in \r^2 \mid x \leq 0\}) &= {\rm conv}((0,-1),(-2,-1),(0,1)). 
\end{align*}
Hence it follows that
\[
\varphi_{w,F}(P^\ast)={\rm conv}((0,1),(-2,-1),(2,-1)).
\]

Therefore, we see that ${\rm mut}_w(P,F)^\ast = \varphi_{w,F}(P^\ast)$ as in \cref{prop:compatibility}; see also Figure \ref{f:comb_mut_example}.
\begin{figure}[!ht]
\begin{center}
   \includegraphics[width=13.0cm,bb=60mm 150mm 150mm 220mm,clip]{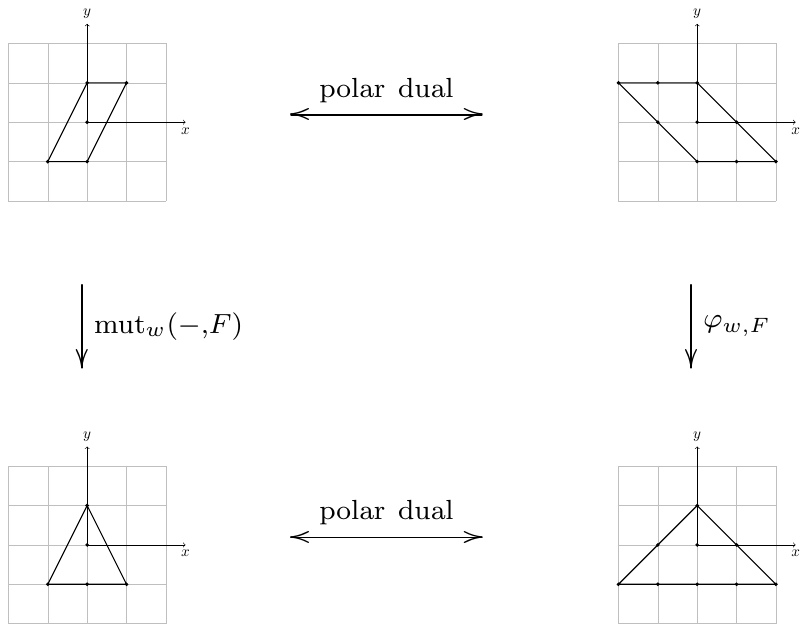}
	\caption{We can see that ${\rm mut}_w(P,F)^\ast = \varphi_{w,F}(P^\ast)$.} 
   \label{f:comb_mut_example}
\end{center}
\end{figure}
\end{ex}

We now introduce the notion of combinatorial mutation equivalence. 

\begin{defi}[{see \cite[Definition 3.5]{Hig}}]
Two lattice polytopes $P$ and $P'$ in $N_\r$ are said to be \textit{combinatorially mutation equivalent in $N_\r$} if there exists a sequence $((w_1,F_1),\ldots,(w_\ell,F_\ell))$, where $w_i \in M$ is primitive and $F_i \subseteq w_i^\perp$ is a lattice polytope, 
such that \[P'={\rm mut}_{w_\ell}((\cdots {\rm mut}_{w_2}({\rm mut}_{w_1}(P,F_1),F_2) \cdots), F_\ell).\] 

Similarly, two rational convex polytopes $Q$ and $Q'$ in $M_\r$ are said to be \textit{combinatorially mutation equivalent in $M_\r$} if there exists a sequence $((w_1,F_1),\ldots,(w_\ell,F_\ell))$, where $w_i \in M$ is primitive and $F_i \subseteq w_i^\perp$ is a lattice polytope, 
such that \[Q'=\varphi_{w_\ell,F_\ell}(\cdots (\varphi_{w_1,F_1}(Q)) \cdots )\]
and the image of each of the intermediate steps is always a rational convex polytope. 
\end{defi}

\subsection{Tropicalized cluster mutations as combinatorial mutations}

In this subsection, we realize the tropicalized cluster mutation $\mu_k ^T$ using the map $\varphi_{w,F}$ in Definition~\ref{def:M_side}.

\begin{prop}\label{p:tropicalized_muation_as_combinatorial}
For $k \in J_{\rm uf}$, the tropicalized cluster mutation $\mu_k ^T \colon \r^J \rightarrow \r^J$ can be described as a composition of a combinatorial mutation in $M_\r$ and $f \in {\it GL}_{J}(\z)$. 
\end{prop}

\begin{proof}
Recall that $\mu_k^T \colon \r^J \rightarrow \r^J, \ (g_j)_{j \in J} \mapsto (g' _j)_{j \in J}$, is defined by 
\begin{align*}
g' _j \coloneqq
\begin{cases}
g_j + [-\varepsilon_{k, j} ^{(t)}]_+ g_k + \varepsilon_{k, j} ^{(t)} [g_k]_+ &(j \neq k),\\
-g_j &(j = k)
\end{cases}
\end{align*}
for $j \in J$, where $(\varepsilon_{i,j}^{(t)})_{i \in J_{\rm uf}, j \in J} \in {\rm Mat}_{J_{\rm uf} \times J}(\z)$ is a full rank matrix whose $J_{\rm uf} \times J_{\rm uf}$-submatrix $\varepsilon^{\circ}$ is skew-symmetrizable. For $i \in J_{\rm uf}$, let $\varepsilon_i^{(t)}$ denote the $i$-th row of the matrix $(\varepsilon_{i,j}^{(t)})_{i \in J_{\rm uf},j \in J}$. 

For $j \in J$, we write the $j$-th unit vector of $\r^J$ as ${\bm e}_j \in \r^J$. Define ${\bm u}_k = (u_{k, j})_j \in \z^J$ by 
\begin{align*}
u_{k, j} \coloneqq \begin{cases}
\min\{\varepsilon_{k,j}^{(t)},0\} &(j \neq k), \\
2 &(j=k). 
\end{cases}
\end{align*}
Let $f \colon \r^J \rightarrow \r^J$ be a linear map defined by the matrix $(f_{i,j})_{i,j \in J}$ whose $i$-th row is ${\bm e}_i$ if $i \neq k$ and ${\bm e}_k-{\bm u}_k$ if $i=k$, 
where $f$ acts on $g \in \r^J$ from the right, that is, we regard $g$ as a row vector. 
Then we notice that $f \in {\it GL}_{J}(\z)$. Let us write $w \coloneqq \frac{1}{c_k^{(t)}}\varepsilon_k^{(t)} \in \z^J$, 
where $c_k^{(t)}$ is the greatest common divisor of the absolute values of the nonzero entries of $\varepsilon_k^{(t)}$. We set $F \coloneqq {\rm conv}({\bf 0},-c_k^{(t)}{\bm e}_k)$, where ${\bf 0}$ denotes the origin of $\r^J$. 
Notice that $w$ is primitive, and $F \subseteq w^\perp$ since the $k$-th entry of $\varepsilon_k^{(t)}$ is $0$, which follows from the skew-symmetrizability of $\varepsilon^{\circ}$. 

Our goal is to show that $\mu_k^T=f \circ \varphi_{w,F}$ as maps. 
For $g = (g_j)_{j \in J} \in \r^J$, the direct computation shows that 
\begin{align*}
\varphi_{w,F}(g)=g-\min\{ \langle g, v \rangle \mid v \in F\}w 
=g-\min\{0, -g_k\}\varepsilon^{(t)}_k 
=\begin{cases} 
g &\text{if }g_k \leq 0, \\
g+g_k\varepsilon_k^{(t)} &\text{if }g_k \geq 0. 
\end{cases}
\end{align*}
Moreover, we see the following: 
\begin{align*}
f(g)&=g-g_k {\bm u}_k = (g_j')_{j \in J}, \text{ where }\\
g_j'&=\begin{cases}
g_j-g_k \min\{\varepsilon_{k,j}^{(t)},0\}=g_j+[-\varepsilon_{k,j}^{(t)}]_+g_k &(j \neq k), \\
g_j-2g_j=-g_j &(j=k), 
\end{cases}
\end{align*}
which coincides with $\mu_k^T(g)$ in the case $g_k \leq 0$. Similar to this, we obtain the following: 
\begin{align*}
f(g+g_k\varepsilon_k^{(t)})&=g+g_k\varepsilon_k^{(t)}-g_k {\bm u}_k=(g_j')_{j \in J}, \text{ where} \\
g_j'&=\begin{cases}
g_j+g_k\varepsilon_{k,j}^{(t)}-g_k\min\{\varepsilon_{k,j}^{(t)},0\}=g_j+[-\varepsilon_{k,j}^{(t)}]_+g_k+\varepsilon_{k,j}^{(t)} g_k &(j \neq k), \\
g_j+0-2g_j=-g_j &(j=k), 
\end{cases}
\end{align*}
which coincides with $\mu_k^T(g)$ in the case $g_k \geq 0$. This proves the proposition.
\end{proof}

\bigskip

\section{Case of flag and Schubert varieties}\label{s:flag_Schubert}

In this section, we restrict ourselves to the case of flag varieties and Schubert varieties. 
In Section \ref{ss:string_NZ}, we review fundamental properties of these varieties, and recall basic facts on their Newton--Okounkov bodies. 
In Section \ref{ss:Schubert_as_cluster}, we review results of \cite{FO2}, which connect string polytopes and Nakashima--Zelevinsky polytopes by tropicalized cluster mutations.
In Section \ref{ss:combinatorial_mutations_on_NO}, we study relations among Newton--Okounkov bodies of Schubert varieties arising from cluster structures from the viewpoint of combinatorial mutations.

\subsection{String polytopes and Nakashima--Zelevinsky polytopes}\label{ss:string_NZ}

Let $G$ be a connected, simply-connected semisimple algebraic group over $\mathbb{C}$, $\mathfrak{g}$ the Lie algebra of $G$, and $B$ a Borel subgroup of $G$. 
Then the quotient space $G/B$ is called the \emph{full flag variety}, which is a nonsingular projective variety. 
Choose a maximal torus $H \subseteq B$, and denote by $\mathfrak{h} \subseteq \mathfrak{g}$ the Lie algebra of $H$. 
Let $\mathfrak{h}^\ast \coloneqq {\rm Hom}_\mathbb{C} (\mathfrak{h}, \mathbb{C})$ be the dual space of $\mathfrak{h}$, $\langle \cdot, \cdot \rangle \colon \mathfrak{h}^\ast \times \mathfrak{h} \rightarrow \mathbb{C}$ the canonical pairing, $P \subseteq \mathfrak{h}^\ast$ the weight lattice for $\mathfrak{g}$, and $P_+ \subseteq P$ the set of dominant integral weights. 
For $\lambda \in P$, there uniquely exists a character $\tilde{\lambda} \colon H \rightarrow \c^\times$ of $H$ such that $d \tilde{\lambda} = \lambda$.
By composing this with the canonical projection $B \twoheadrightarrow H$, we obtain a group homomorphism $\tilde{\lambda} \colon B \rightarrow \c^\times$, which we also denote by $\tilde{\lambda}$.
We consider the Cartan matrix 
\[C(\mathfrak{g}) = (c_{i,j})_{i, j \in I} \coloneqq (\langle \alpha_j, h_i \rangle)_{i, j \in I}\] 
of $\mathfrak{g}$, where $\{\alpha_i \mid i \in I\} \subseteq P$ (resp., $\{h_i \mid i \in I\} \subseteq \mathfrak{h}$) denotes the set of simple roots (resp., simple coroots).
Let $e_i, f_i, h_i \in \mathfrak{g}$, $i \in I$, denote the Chevalley generators, $N_G(H)$ the normalizer of $H$ in $G$, and $W \coloneqq N_G(H)/H$ the Weyl group.
The group $W$ is generated by the set $\{s_i \mid i \in I\}$ of simple reflections. 
We call ${\bm i} = (i_1, \ldots, i_m) \in I^m$ a \emph{reduced word} for $w \in W$ if $w = s_{i_1} \cdots s_{i_m}$ and if $m$ is the minimum among such expressions of $w$. 
In this case, the length $m$ is called the \emph{length} of $w$, which is denoted by $\ell(w)$. Denote by $R(w)$ the set of reduced words for $w$.

\begin{defi}[{see, for instance, \cite[Section I\hspace{-.1em}I.13.3]{Jan} and \cite[Definition 7.1.13]{Kum}}]\normalfont
For $w \in W$, we define a closed subvariety $X(w)$ of $G/B$ to be the Zariski closure of $B \widetilde{w} B/B$ in $G/B$, where $\widetilde{w} \in N_G(H)$ is a lift for $w \in W = N_G(H)/H$. This variety $X(w)$ is called a \emph{Schubert variety}.
\end{defi} 

For $w \in W$, the Schubert variety $X(w)$ is a normal projective variety, and we have $\dim_\c (X(w)) = \ell(w)$ (see, for instance, \cite[Sects.~I\hspace{-.1em}I.13.3, I\hspace{-.1em}I.14.15]{Jan}). 
Denoting the longest element in $W$ by $w_0$, we see that the Schubert variety $X(w_0)$ coincides with $G/B$. 
For $\lambda \in P_+$, define a line bundle $\mathcal{L}_\lambda$ on $G/B$ by 
\[
\mathcal{L}_\lambda \coloneqq (G \times \mathbb{C})/B,
\] 
where the right $B$-action on $G \times \mathbb{C}$ is defined as follows: 
\[(g, c) \cdot b \coloneqq (g b, \tilde{\lambda}(b) c)\] 
for $g \in G$, $c \in \mathbb{C}$, and $b \in B$. 
Let us denote the restriction of $\mathcal{L}_\lambda$ to $X(w)$ by the same symbol $\mathcal{L}_\lambda$. 
By \cite[Proposition 1.4.1]{Bri}, the line bundle $\mathcal{L}_\lambda$ on $X(w)$ is generated by global sections. 
Let $\mathcal{O}(K_{G/B})$ denote the canonical bundle of $G/B$. By \cite[Proposition 2.2.7 (ii)]{Bri}, we have 
\[
\mathcal{O}(K_{G/B}) \simeq \mathcal{L}_{-2\rho},
\] 
where $\rho \in P_+$ denotes the half sum of the positive roots. For $\lambda \in P_+$, denote by $V(\lambda)$ the irreducible highest weight $G$-module over $\c$ with highest weight $\lambda$ and with highest weight vector $v_{\lambda}$. 
For $w \in W$ and $\lambda \in P_+$, we define a $B$-submodule $V_w(\lambda)$ of $V(\lambda)$, called a \emph{Demazure module}, by 
\[
V_w(\lambda) \coloneqq \sum_{b \in B} \mathbb{C} b \widetilde{w} v_{\lambda},
\] 
where $\widetilde{w} \in N_G(H)$ is a lift for $w$. 
By the Borel--Weil type theorem (see, for instance, \cite[Corollary 8.1.26]{Kum}), it follows that the space $H^0(G/B, \mathcal{L}_\lambda)$ (resp., $H^0(X(w), \mathcal{L}_\lambda)$) of global sections is a $G$-module (resp., a $B$-module) isomorphic to the dual module $V(\lambda)^\ast \coloneqq {\rm Hom}_\mathbb{C}(V(\lambda), \mathbb{C})$ (resp., $V_w (\lambda)^\ast \coloneqq {\rm Hom}_\mathbb{C}(V_w(\lambda), \mathbb{C})$). 
Let us fix a lowest weight vector $\tau_\lambda \in H^0(G/B, \mathcal{L}_\lambda)$, and consider its restriction in $H^0(X(w), \mathcal{L}_\lambda)$, which is also denoted by the same symbol $\tau_\lambda$. 
For ${\bm i} \in R(w)$ and $\lambda \in P_+$, we denote by $\Delta_{\bm i} (\lambda)$ (resp., $\widetilde{\Delta}_{\bm i} (\lambda)$) the corresponding string polytope (resp., the corresponding Nakashima--Zelevinsky polytope); see \cite[Section 1]{Lit}, \cite[Definition 2.15]{FN}, \cite[Definition 3.24]{FO1}, and \cite[Definition 3.9]{Fuj} for the precise definitions.
These polytopes are defined from \emph{Kashiwara's crystal basis} which is a combinatorial skeleton of $V(\lambda)$; see \cite{Kas} for a survey on crystal bases.
Those who are not familiar with crystal bases may regard \cref{t:NO_body_crystal_lowest_term_valuations} below as their definitions.

\begin{rem}
The definition of Nakashima--Zelevinsky polytopes in \cite[Definition 3.9]{Fuj} is slightly different from the one in \cite[Definition 3.24 (2)]{FO1}; they coincide after reversing the order of coordinates. In the present paper, let us use the definition in \cite[Definition 3.9]{Fuj}.
\end{rem}

Kaveh \cite{Kav} proved that the string polytope $\Delta_{\bm i}(\lambda)$ is identical to the Newton--Okounkov body $\Delta(X(w), \mathcal{L}_\lambda, v_{\bm i} ^{\rm high}, \tau_\lambda)$ of $(X(w), \mathcal{L}_\lambda)$ associated with a highest term valuation $v_{\bm i} ^{\rm high}$. 
Using a different kind of highest term valuation $\tilde{v}_{\bm i} ^{\rm high}$, the first named author and Naito \cite{FN} showed that the Nakashima--Zelevinsky polytope $\widetilde{\Delta}_{\bm i}(\lambda)$ can be realized as a Newton--Okounkov body $\Delta(X(w), \mathcal{L}_\lambda, \tilde{v}_{\bm i} ^{\rm high}, \tau_\lambda)$. 
Afterward, the first named author and Oya \cite{FO1} gave different realizations of these polytopes as Newton--Okounkov bodies, which we review in the following.
For each ${\bm i} = (i_1, \ldots, i_m) \in R(w)$, we obtain a birational morphism 
\[\c^m \rightarrow X(w),\ (t_1, \ldots, t_m) \mapsto \exp(t_1 f_{i_1}) \cdots \exp(t_m f_{i_m}) \bmod B,\] 
by \cite[Chapter I\hspace{-.1em}I.13]{Jan}. 
Then the function field $\mathbb{C}(X(w))$ of $X(w)$ is identified with the field $\c(t_1, \ldots, t_m)$ of rational functions in $t_1, \ldots, t_m$. 
We define two total orders $\leq$ and $\preceq$ on $\z^m$ as follows: 
\begin{align*}
(a_1, \ldots, a_m) < (a_1 ^\prime, \ldots, a_m ^\prime) & & ({\rm resp.},\ (a_1, \ldots, a_m) \prec (a_1 ^\prime, \ldots, a_m ^\prime))
\end{align*}
if and only if there exists $1 \leq k \leq m$ such that 
\begin{align*}
a_1 = a_1 ^\prime, \ldots, a_{k-1} = a_{k-1} ^\prime,\ a_k < a_k ^\prime & & ({\rm resp.},\ a_m = a_m ^\prime, \ldots, a_{k+1} = a_{k+1} ^\prime,\ a_k < a_k ^\prime).
\end{align*}
Then let $v_{\bm i} ^{\rm low}$ and $\tilde{v}_{\bm i} ^{\rm low}$ denote the valuations on $\c(X(w)) \simeq \c(t_1, \ldots, t_m)$ defined to be $v^{\rm low} _{\leq}$ and $v^{\rm low} _{\preceq}$ on $\c(t_1, \ldots, t_m)$, respectively (see \cref{ex:lowest_term_valuation}). 

\begin{thm}[{see \cite[Propositions 3.28, 3.29 and Corollary 5.3]{FO1}}]\label{t:NO_body_crystal_lowest_term_valuations}
For $w \in W$, $\lambda \in P_+$, and ${\bm i} \in R(w)$, the following hold.
\begin{enumerate}
\item[{\rm (1)}] The Newton--Okounkov body $\Delta(X(w), \mathcal{L}_\lambda, \tilde{v}_{\bm i} ^{\rm low}, \tau_\lambda)$ coincides with the string polytope $\Delta_{\bm i} (\lambda)$.
\item[{\rm (2)}] The Newton--Okounkov body $\Delta(X(w), \mathcal{L}_\lambda, v_{\bm i} ^{\rm low}, \tau_\lambda)$ coincides with the Nakashima--Zelevinsky polytope $\widetilde{\Delta}_{\bm i} (\lambda)$.
\end{enumerate}
\end{thm}

\cref{t:NO_body_crystal_lowest_term_valuations} leads to the realization of $\Delta_{\bm i}(\lambda)$ (resp., $\widetilde{\Delta}_{\bm i}(\lambda)$) in \cite{FO2} as a Newton--Okounkov body arising from a cluster structure, which is reviewed in the next subsection.

In the context of mirror symmetry, when $G$ is of type $A_n$, Rusinko \cite[Theorem 7]{Rus} proved that the polar dual of the (properly translated) string polytope $\Delta_{\bm i} (2 \rho)$ is a lattice polytope for all ${\bm i} \in R(w_0)$. Using Hibi's criterion \cite{Hib} on the integrality of the vertices of the dual polytopes, Steinert \cite{Ste} generalized this result to all Lie types and all valuations $v \colon \mathbb{C}(G/B) \setminus \{0\} \rightarrow \z^{\dim_\c (G/B)}$ with $1$-dimensional leaves as follows.

\begin{thm}[{see \cite[Sects.\ 4, 6]{Ste}}]\label{t:unique_lattice_point_Steinert}
Take a valuation $v \colon \mathbb{C}(G/B) \setminus \{0\} \rightarrow \z^{\dim_\c (G/B)}$ with $1$-dimensional leaves, and fix a nonzero section $\tau \in H^0 (G/B, \mathcal{L}_{2 \rho})$. If the semigroup $S(G/B, \mathcal{L}_{2 \rho}, v, \tau_{2 \rho})$ is finitely generated and saturated, then the Newton--Okounkov body $\Delta(G/B, \mathcal{L}_{2 \rho}, v, \tau_{2 \rho})$ contains exactly one lattice point in its interior. In addition, the dual $\Delta(G/B, \mathcal{L}_{2 \rho}, v, \tau_{2 \rho})^\vee$ in the sense of Section \ref{s:intro} is a lattice polytope.
\end{thm} 

In \cref{t:combinatorial_mutations_NO} (2), we study combinatorial mutations on the lattice polytope $\Delta(G/B, \mathcal{L}_{2 \rho}, v, \tau_{2 \rho})^\vee$ when $v$ comes from the cluster structure.

\begin{rem}
In the paper \cite{Ste}, the algebraic group $G$ is assumed to be simple. 
However, the proof of \cref{t:unique_lattice_point_Steinert} can also be applied to the case that $G$ is semisimple.
When $G$ is simple, it is proved in [47] that $\Delta(G/B, \mathcal{L}_\lambda, v, \tau_\lambda)$ contains exactly one lattice point in its interior only if $\lambda = 2 \rho$.
\end{rem}

\begin{ex}\label{ex:unique_lattice_point}
Assume that $G$ is simple and simply-laced. We identify the set $I$ of vertices of the Dynkin diagram with $\{1, 2, \ldots, n\}$ as follows:
\begin{align*}
&A_n\ \begin{xy}
\ar@{-} (50,0) *++!D{1} *\cir<3pt>{};
(60,0) *++!D{2} *\cir<3pt>{}="C"
\ar@{-} "C";(65,0) \ar@{.} (65,0);(70,0)^*!U{}
\ar@{-} (70,0);(75,0) *++!D{n-1} *\cir<3pt>{}="D"
\ar@{-} "D";(85,0) *++!D{n} *\cir<3pt>{}="E"
\end{xy}\hspace{-1mm},&D_n\ \begin{xy}
\ar@{-} (60,0) *++!D{3} *\cir<3pt>{}="C";
(65,0)
\ar@{.} (65,0);(70,0)^*!U{}
\ar@{-} (70,0);(75,0) *++!D{n-1} *\cir<3pt>{}="D"
\ar@{-} "D";(85,0) *++!D{n} *\cir<3pt>{}
\ar@{-} "C";(50,4) *++!R{1} *\cir<3pt>{}
\ar@{-} "C";(50,-4) *++!R{2} *\cir<3pt>{},
\end{xy}\hspace{-1mm},\\ 
&E_6\ \begin{xy}
\ar@{-} (50,0) *++!U{5} *\cir<3pt>{};
(60,0) *++!U{4} *\cir<3pt>{}="C"
\ar@{-} "C";(70,0) *++!U{3} *\cir<3pt>{}="D"
\ar@{-} "D";(80,0) *++!U{2} *\cir<3pt>{}="E"
\ar@{-} "E";(90,0) *++!U{6} *\cir<3pt>{}="F"
\ar@{-} "D";(70,10) *++!D{1} *\cir<3pt>{}="G"
\end{xy}\hspace{-1mm},&E_7\ \begin{xy}
\ar@{-} (50,0) *++!U{5} *\cir<3pt>{};
(60,0) *++!U{4} *\cir<3pt>{}="C"
\ar@{-} "C";(70,0) *++!U{3} *\cir<3pt>{}="D"
\ar@{-} "D";(80,0) *++!U{2} *\cir<3pt>{}="E"
\ar@{-} "E";(90,0) *++!U{6} *\cir<3pt>{}="F"
\ar@{-} "D";(70,10) *++!D{1} *\cir<3pt>{}="G"
\ar@{-} "F";(100,0) *++!U{7} *\cir<3pt>{}="H"
\end{xy}\hspace{-1mm},\\ 
&E_8\ \begin{xy}
\ar@{-} (50,0) *++!U{5} *\cir<3pt>{};
(60,0) *++!U{4} *\cir<3pt>{}="C"
\ar@{-} "C";(70,0) *++!U{3} *\cir<3pt>{}="D"
\ar@{-} "D";(80,0) *++!U{2} *\cir<3pt>{}="E"
\ar@{-} "E";(90,0) *++!U{6} *\cir<3pt>{}="F"
\ar@{-} "D";(70,10) *++!D{1} *\cir<3pt>{}="G"
\ar@{-} "F";(100,0) *++!U{7} *\cir<3pt>{}="H"
\ar@{-} "H";(110,0) *++!U{8} *\cir<3pt>{}="I"
\end{xy}. &
\end{align*}
Let $X_n$ be the Lie type of $G$, and define ${\bm i}_{X_n} \in R(w_0)$ as follows. 
\begin{itemize}
\item If $G$ is of type $A_n$, then
\[{\bm i}_{A_n} \coloneqq (1, 2, 1, 3, 2, 1, \ldots, n, n-1, \ldots, 1) \in I^{\frac{n(n+1)}{2}}.\]
\item If $G$ is of type $D_n$, then
\[{\bm i}_{D_n} \coloneqq (1, 2, \underbrace{3, 1, 2, 3}_4, \underbrace{4, 3, 1, 2, 3, 4}_6, \ldots, \underbrace{n, n-1, \ldots, 3, 1, 2, 3, \ldots, n-1, n}_{2n-2}) \in I^{n (n-1)}.\] 
\item If $G$ is of type $E_6$, then
\[{\bm i}_{E_6} \coloneqq ({\bm i}_{D_5}, 6, 2, 3, 1, 4, 5, 3, 4, 2, 3, 1, 6, 2, 3, 4, 5) \in I^{36}.\] 
\item If $G$ is of type $E_7$, then
\[{\bm i}_{E_7} \coloneqq ({\bm i}_{E_6}, 7, 6, 2, 3, 1, 4, 5, 3, 4, 2, 3, 1, 6, 2, 3, 4, 5, 7, 6, 2, 3, 1, 4, 3, 2, 6, 7) \in I^{63}.\] 
\item If $G$ is of type $E_8$, then
\begin{align*}
{\bm i}_{E_8} \coloneqq ({\bm i}_{E_7}, 8&, 7, 6, 2, 3, 1, 4, 5, 3, 4, 2, 3, 1, 6, 2, 3, 4, 5, 7,\\ 
6&, 2, 3, 1, 4, 3, 2, 6, 7, 8, 7, 6, 2, 3, 1, 4, 5, 3, 4,\\ 
2&, 3, 1, 6, 2, 3, 4, 5, 7, 6, 2, 3, 1, 4, 3, 2, 6, 7, 8) \in I^{120}.
\end{align*}
\end{itemize}
Littelmann \cite[Section 1, Corollaries 4, 8, and Theorems 8.1, 8.2, 9.3]{Lit} gave a system of explicit affine inequalities defining the string polytopes $\Delta_{\bm i} (\lambda)$ associated with the reduced words ${\bm i}$ above and $\lambda \in P_+$. 
When $\lambda = 2 \rho$, let ${\bm a}_{X_n}$ denote the unique lattice point of $\Delta_{{\bm i}_{X_n}} (2 \rho)$. Then we see the following by Littelmann's description.
\begin{itemize}
\item If $G$ is of type $A_n$, then
\[{\bm a}_{A_n} = (1, 2, 1, 3, 2, 1, \ldots, n, n-1, \ldots, 1) \in \z^{\frac{n(n+1)}{2}}.\]
\item If $G$ is of type $D_n$, then
\[{\bm a}_{D_n} = (1, 1, \underbrace{3, 2, 2, 1}_4, \underbrace{5, 4, 3, 3, 2, 1}_6, \ldots, \underbrace{2n -3, 2n-4, \ldots, n, n-1, n-1, n-2, \ldots, 2, 1}_{2n-2}) \in \z^{n (n-1)}.\] 
\item If $G$ is of type $E_6$, then
\[{\bm a}_{E_6} = ({\bm a}_{D_5}, 11, 10, 9, 8, 8, 7, 7, 6, 6, 5, 4, 5, 4, 3, 2, 1) \in \z^{36}.\] 
\item If $G$ is of type $E_7$, then
\[{\bm a}_{E_7} = ({\bm a}_{E_6}, 17, 16, 15, 14, 13, 13, 12, 12, 11, 11, 10, 9, 10, 9, 8, 7, 6, 9, 8, 7, 6, 5, 5, 4, 3, 2, 1) \in \z^{63}.\] 
\item If $G$ is of type $E_8$, then
\begin{align*}
{\bm a}_{E_8} = ({\bm a}_{E_7}, 28&, 27, 26, 25, 24, 23, 23, 22, 22, 21, 21, 20, 19, 20, 19, 18, 17, 16, 19,\\ 
18&, 17, 16, 15, 15, 14, 13, 12, 11, 29, 18, 17, 16, 15, 14, 14, 13, 13, 12,\\ 
12&, 11, 10, 11, 10, 9, 8, 7, 10, 9, 8, 7, 6, 6, 5, 4, 3, 2, 1) \in \z^{120}.
\end{align*}
\end{itemize}
\end{ex}

\begin{ex}[{see \cite[Corollary 5]{Lit} and \cite[Theorem 6.1]{Nak}}]\label{ex:GT_polytopes_type_A}
Let $G = SL_{n+1} (\c)$, and $\lambda \in P_+$. We consider the reduced word ${\bm i}_{A_n} \in R(w_0)$ in \cref{ex:unique_lattice_point}. Then the string polytope $\Delta_{{\bm i}_{A_n}}(\lambda)$ and the Nakashima--Zelevinsky polytope $\widetilde{\Delta}_{{\bm i}_{A_n}}(\lambda)$ are both unimodularly equivalent to the Gelfand--Tsetlin polytope $GT(\lambda)$ which is defined to be the set of 
\[
(a_1 ^{(1)}, a_1 ^{(2)}, a_2 ^{(1)}, a_1 ^{(3)}, a_2 ^{(2)}, a_3 ^{(1)}, \ldots, a_1 ^{(n)}, \ldots, a_n ^{(1)}) \in \r^{\frac{n(n +1)}{2}}
\] 
satisfying the following conditions:
\[\begin{matrix}
\lambda_{\geq 1} & & \lambda_{\geq 2} & & \cdots & & & \lambda_{\geq n} & & 0\\
 & a_1 ^{(1)} & & a_2 ^{(1)} & & \cdots & & & a_n ^{(1)} & \\
 & & a_1 ^{(2)} & & \cdots & & & a_{n -1} ^{(2)} & & \\
 & & & \ddots & & \ldots & & & & \\
 & & & & a_1 ^{(n-1)} & & a_2 ^{(n-1)} & & & \\
 & & & & & a_1 ^{(n)}, & & & & 
\end{matrix}\]
where $\lambda_{\geq k} \coloneqq \sum_{k \leq \ell \leq n} \langle \lambda, h_\ell \rangle$ for $1 \leq k \leq n$, and the notation 
\[\begin{matrix}
a & & c\\
 & b & 
\end{matrix}\]
means that $a \geq b \geq c$. 
\end{ex}

\begin{ex}[{see \cite[Corollary 7]{Lit}}]\label{ex:GT_polytopes_type_C}
Let $G = Sp_{2n} (\c)$, and $\lambda \in P_+$. 
We identify the set $I$ of vertices of the Dynkin diagram with $\{1, 2, \ldots, n\}$ as follows:
\begin{align*}
&C_n\ \begin{xy}
\ar@{=>} (50,0) *++!D{1} *\cir<3pt>{};
(60,0) *++!D{2} *\cir<3pt>{}="C"
\ar@{-} "C";(65,0) \ar@{.} (65,0);(70,0)^*!U{}
\ar@{-} (70,0);(75,0) *++!D{n-1} *\cir<3pt>{}="D"
\ar@{-} "D";(85,0) *++!D{n} *\cir<3pt>{}="E"
\end{xy}.
\end{align*}
Define ${\bm i}_{C_n} \in R(w_0)$ by
\[
{\bm i}_{C_n} \coloneqq (1, \underbrace{2, 1, 2}_3, \underbrace{3, 2, 1, 2, 3}_5, \ldots, \underbrace{n, n-1, \ldots, 1, \ldots, n-1, n}_{2n-1}) \in I^{n^2}.
\] 
Then the string polytope $\Delta_{{\bm i}_{C_n}}(\lambda)$ is unimodularly equivalent to the Gelfand--Tsetlin polytope $GT_{C_n}(\lambda)$ of type $C_n$ which is defined to be the set of 
\[
(a_1 ^{(1)}, \underbrace{b_1 ^{(2)}, a_2 ^{(1)}, a_1 ^{(2)}}_3, \underbrace{b_1 ^{(3)}, b_2 ^{(2)},  a_3 ^{(1)},  a_2 ^{(2)},  a_1 ^{(3)}}_5, \ldots, \underbrace{b_1 ^{(n)}, \ldots, b_{n-1} ^{(2)}, a_n ^{(1)}, \ldots, a_1 ^{(n)}}_{2n-1}) \in \r^{n^2}
\] 
satisfying the following conditions as in \cref{ex:GT_polytopes_type_A}:
\[\begin{matrix}
\lambda_{\leq n} & & \lambda_{\leq n-1} & & \cdots & \lambda_{\leq 1} & & 0 \\
 & a_1 ^{(1)} & & a_2 ^{(1)} & & & a_n ^{(1)} & \\
 & & b_1 ^{(2)} & & \cdots & b_{n -1} ^{(2)}  & & 0\\
 & & & a_1 ^{(2)} & & & a_{n -1} ^{(2)} & \\
 & & & & \ddots & & & \vdots \\
 & & & & & b_1 ^{(n)} & & 0\\
 & & & & & & a_1 ^{(n)}, &
\end{matrix}\]
where $\lambda_{\leq k} \coloneqq \sum_{1 \leq \ell \leq k} \langle \lambda, h_\ell \rangle$ for $1 \leq k \leq n$.
\end{ex}

\begin{ex}[{see \cite[Example 5.10]{FN}}]
Let $G = Sp_4 (\c)$, and $\lambda \in P_+$. 
We define ${\bm i}_{C_2} \in R(w_0)$ as in \cref{ex:GT_polytopes_type_C}.
Then the Nakashima--Zelevinsky polytope $\widetilde{\Delta}_{{\bm i}_{C_2}}(\lambda)$ coincides with the set of $(a_1, \ldots, a_4) \in \r^4_{\geq 0}$ satisfying the following inequalities:
\begin{align*}
a_4 \leq \langle \lambda, h_2 \rangle,\quad a_3 \leq a_4 + \langle \lambda, h_1 \rangle,\quad a_2 \leq \min\{a_3 + \langle \lambda, h_1 \rangle, 2a_3\},\quad 2a_1 \leq \min\{2\langle \lambda, h_1 \rangle, a_2\}.
\end{align*}
If $\langle \lambda, h_i \rangle > 0$ for $1 \leq i \leq 2$, then $\widetilde{\Delta}_{{\bm i}_{C_2}}(\lambda)$ is a $4$-dimensional polytope with $11$ vertices. 
This is not unimodularly equivalent to the Gelfand--Tsetlin polytope $GT_{C_2}(\lambda)$ of type $C_2$ since $GT_{C_2}(\lambda)$ has $12$ vertices.
\end{ex}

\subsection{Newton--Okounkov bodies of Schubert varieties arising from cluster structures}\label{ss:Schubert_as_cluster}

In this subsection, we assume that $G$ is simply-laced. 
Let $B^-$ be the Borel subgroup of $G$ opposite to $B$, and $U^-$ the unipotent radical of $B^-$. Then we obtain an open embedding
\[U^- \hookrightarrow G/B,\ u \mapsto u \bmod B.\] 
For $w \in W$, set 
\begin{align*}
U^-_w&\coloneqq U^-\cap B \widetilde{w} B \subseteq G,
\end{align*}
where we take a lift $\widetilde{w} \in N_G(H)$ for $w \in W = N_G(H)/H$. 
We call $U^-_w$ the \emph{unipotent cell} associated with $w$. 
Note that the open embedding $U^- \hookrightarrow G/B$ above induces an open embedding $U_w^- \hookrightarrow X(w)$ and an isomorphism $\c(U_w^-) \simeq \c(X(w))$.
For $i \in I$, let $\mathfrak{g}_i \subseteq \mathfrak{g}$ denote the Lie subalgebra generated by $e_i, f_i, h_i$, which is isomorphic to $\mathfrak{sl}_2 (\c)$.
Then the embedding $\mathfrak{g}_i \hookrightarrow \mathfrak{g}$ of Lie algebras naturally lifts to an algebraic group homomorphism $\varphi_i \colon SL_2(\c) \rightarrow G$. 
For $(i_1, \ldots, i_m) \in R(w)$, we define a lift $\overline{w} \in N_G(H)$ for $w \in W = N_G(H)/H$ by 
\[\overline{w} \coloneqq \overline{s}_{i_1} \cdots \overline{s}_{i_m},\]
where we write 
\[\overline{s}_i \coloneqq \varphi_i \left(\begin{pmatrix}
0 & -1\\
1 & 0
\end{pmatrix}\right)\]
for $i \in I$. 
The element $\overline{w}$ is independent of the choice of a reduced word $(i_1, \ldots, i_m) \in R(w)$. 
For $w \in W$ and $\lambda \in P_+$, we set $v_{w \lambda} \coloneqq \overline{w} v_\lambda \in V(\lambda)$, and define $f_{w \lambda} \in V(\lambda)^{\ast}$ by $f_{w \lambda} (v_{w \lambda}) = 1$ and by $f(v) = 0$ for each weight vector $v \in V(\lambda)$ whose weight is different from $w \lambda$. 
Let $\{\varpi_i \mid i \in I\} \subseteq  P_+$ be the set of fundamental weights. 
For $u, u^\prime \in W$ and $i \in I$, we define a function $\Delta_{u \varpi_i, u^\prime \varpi_i} \in \c[G]$ by 
\[\Delta_{u \varpi_i, u^\prime \varpi_i} (g) \coloneqq f_{u \varpi_i} (g v_{u^\prime \varpi_i}),\]
which is called a \emph{generalized minor} (see also \cite[Section 2.3]{BFZ}). 
Denote by $D_{u \varpi_i, u^\prime \varpi_i} \in \c[U_w^-]$ the restriction of $\Delta_{u \varpi_i, u^\prime \varpi_i}$ to $U_w^-$, which is called a \emph{unipotent minor}.

Fix $w \in W$ and ${\bm i} = (i_1, \ldots, i_m) \in R(w)$. 
For $1 \leq k \leq m$, we write 
\begin{align*}
&w_{\leq k} \coloneqq s_{i_1} \cdots s_{i_k},\ {\rm and}\\
&k^{+} \coloneqq \min(\{k+1 \leq j \leq m \mid i_j = i_k\} \cup \{m+1\}).
\end{align*} 
We set 
\[
J \coloneqq \{1, \ldots, m\},\ J_{\rm fr} \coloneqq \{j \in J \mid j^+ = m+1\},\ J_{\rm uf} \coloneqq J\setminus J_{\rm fr},
\]
and define an integer matrix $\varepsilon^{\bm i} = (\varepsilon_{s, t})_{s \in J_{\rm uf}, t \in J} \in {\rm Mat}_{J_{\rm uf} \times J}(\mathbb{Z})$ as follows:
\[
\varepsilon_{s, t} \coloneqq
	\begin{cases}
   1&\text{if}\ s^+ = t, \\
	-1&\text{if}\ s = t^+, \\
   	c_{i_t, i_s}&\text{if}\ s < t < s^+ < t^+, \\
	-c_{i_t, i_s}&\text{if}\ t < s < t^+ < s^+,\\
	0&\text{otherwise},
\end{cases}
\]
where we recall that $(c_{i, j})_{i, j \in I}$ is the Cartan matrix.

\begin{ex}
Let $G = SL_4 (\c)$, and ${\bm i} = (1, 2, 1, 3, 2, 1) \in R(w_0)$.
Then we have 
\[\varepsilon^{\bm i} = \begin{pmatrix}
0 & -1 & 1 & 0 & 0 & 0 \\
1 & 0 & -1 & -1 & 1 & 0 \\
-1 & 1 & 0 & 0 & -1 & 1 
\end{pmatrix},\]
which coincides with the matrix $\varepsilon$ discussed in \cref{ex:basic_example_cluster_mutation}.
\end{ex}

Let us set  
\[
D(s, {\bm i}) \coloneqq D_{w_{\leq s} \varpi_{i_s}, \varpi_{i_s}}
\]
for $s \in J$.
We consider the cluster pattern $\mathcal{S} = \{{\mathbf s}_t = (\mathbf{A}_t, \varepsilon_t)\}_{t \in \mathbb{T}}$ whose initial seed is given as $\mathbf{s}_{t_0} = ((A_{s; t_0})_{s \in J}, \varepsilon^{\bm i})$. 

\begin{thm}[{\cite[Theorem 2.10]{BFZ} (see also \cite[Theorem B.4]{FO2})}]\label{t:upperBruhat}
For $w \in W$ and ${\bm i} = (i_1, \ldots, i_m) \in R(w)$, there exists a $\mathbb{C}$-algebra isomorphism $\mathscr{U}(\mathcal{S}) \xrightarrow{\sim} \mathbb{C}[U^-_w]$ given by $A_{s; t_0} \mapsto D(s, {\bm i})$ for each $s \in J$.
\end{thm}

We define a seed $\mathbf{s}_{\bm i}$ of $\c (U^-_w)$ by 
\[
\mathbf{s}_{\bm i} \coloneqq (\mathbf{D}_{\bm i} \coloneqq (D(s, \bm{i}))_{s \in J}, \varepsilon^{\bm i}),
\]
which corresponds to $\mathbf{s}_{t_0} = ((A_{s; t_0})_{s \in J}, \varepsilon^{\bm i}) \in \mathcal{S}$ through the isomorphism in \cref{t:upperBruhat}.
Since $G$ is simply-laced, the entries of $\varepsilon^{\bm i}$ are included in $\{-1, 0, 1\}$. 
Hence we can describe $\varepsilon^{\bm i}$ as the quiver whose vertex set is $J$ and whose arrow set is given as $\{s \rightarrow t \mid \varepsilon_{t, s} = 1\ {\rm or}\ \varepsilon_{s, t} = -1\}$.

\begin{ex}
Let $G = SL_4 (\c)$, and ${\bm i} = (1, 2, 1, 3, 2, 1) \in R(w_0)$.
Then the seed $\mathbf{s}_{\bm i} = (\mathbf{D}_{\bm i}, \varepsilon^{\bm i})$ is given as follows: 

	\hfill
	\scalebox{0.7}[0.7]{
		\begin{xy} 0;<1pt,0pt>:<0pt,-1pt>::
			(180,00) *+{D_{w_0 \varpi_3, \varpi_3}} ="1",
			(120,30) *+{D_{s_1 s_2 \varpi_2, \varpi_2}} ="2",
			(240,30) *+{D_{w_0 \varpi_2, \varpi_2}} ="3",
			(60,60) *+{D_{s_1 \varpi_1, \varpi_1}} ="4",
			(180,60) *+{D_{s_1 s_2 s_1 \varpi_1, \varpi_1}} ="5",
			(300,60) *+{D_{w_0 \varpi_1, \varpi_1}} ="6",
			"2", {\ar"1"},
			"4", {\ar"2"},
			"2", {\ar"5"},
			"5", {\ar"3"},
			"3", {\ar"2"},
			"5", {\ar"4"},
			"6", {\ar"5"},
		\end{xy}
	}
	\hfill
	\hfill

\end{ex}

\begin{prop}[{\cite[Remark 2.14]{BFZ}}]
For $w\in W$, the cluster pattern associated with ${\mathbf s}_{\bm i}$ is independent of the choice of ${\bm i}\in R(w)$. In other words, the seeds $\mathbf{s}_{\bm i}$, ${\bm i} \in R(w)$, are all mutually mutation equivalent. 
\end{prop}

\begin{rem}
This proposition can be extended to non-simply-laced case by \cite[Theorem 3.5]{FG:amal}.
\end{rem}

Let $\mathcal{S} = \{{\mathbf s}_t = (\mathbf{A}_t = (A_{j; t})_{j \in J}, \varepsilon_t)\}_{t \in \mathbb{T}}$ be the cluster pattern associated with ${\mathbf s}_{\bm i}$. 
Then we see by Theorems \ref{t:laurentpheno} and \ref{t:upperBruhat} that the function field $\c(U_w^-)$ is isomorphic to the field $\c(A_{j; t} \mid j \in J)$ of rational functions for each $t \in \mathbb{T}$.
Hence, following Section \ref{ss:cluster_algebra}, we obtain a valuation $v_{\mathbf{s}_t}$ on $\c(U_w^-) \simeq \c(X(w))$.

\begin{thm}[{\cite[Corollaries 6.6, 6.25 and Theorem 7.1]{FO2}}]\label{t:NO_body_parametrized_by_seeds}
If $G$ is simply-laced, then the following hold for all $w \in W$, $\lambda \in P_+$, and $t \in \mathbb{T}$.
\begin{enumerate}
\item[{\rm (1)}] The Newton--Okounkov body $\Delta(X(w), \mathcal{L}_\lambda, v_{\mathbf{s}_t}, \tau_\lambda)$ does not depend on the choice of a refinement of the opposite dominance order $\preceq_{\varepsilon_t} ^{\rm op}$.
\item[{\rm (2)}] The Newton--Okounkov body $\Delta(X(w), \mathcal{L}_\lambda, v_{\mathbf{s}_t}, \tau_\lambda)$ is a rational convex polytope.
\item[{\rm (3)}] If $t \overset{k}{\text{---}} t^\prime$, then it holds that
\[\Delta(X(w), \mathcal{L}_\lambda, v_{\mathbf{s}_{t^\prime}}, \tau_\lambda) = \mu_k ^T (\Delta(X(w), \mathcal{L}_\lambda, v_{\mathbf{s}_t}, \tau_\lambda)).\]
\item[{\rm (4)}] The Newton--Okounkov body $\Delta(X(w), \mathcal{L}_\lambda, v_{\mathbf{s}_{\bm i}}, \tau_\lambda)$ is related to the string polytope $\Delta_{\bm i} (\lambda)$ by a unimodular transformation. 
More strongly, the equality 
\[
\Delta_{\bm i} (\lambda) = \Delta(X(w), \mathcal{L}_\lambda, v_{\mathbf{s}_{\bm i}}, \tau_\lambda) M_{\bm i}
\]
holds, where $M_{\bm i} = (d_{s, t})_{s, t \in J} \in {\rm Mat}_{J \times J}(\mathbb{Z})$ is defined by 
\[
d_{s, t} \coloneqq \begin{cases}
\langle s_{i_{t+1}} \cdots s_{i_s} \varpi_{i_s}, h_{i_t} \rangle &{\rm if}\ t \leq s,\\
0 &{\rm if}\ t > s.
\end{cases}
\]
\item[{\rm (5)}] There exists a seed ${\mathbf s}_{\bm i} ^{\rm mut} = (\mathbf{D}_{\bm i} ^{\rm mut}, \varepsilon^{{\bm i}, {\rm mut}}) \in \mathcal{S}$ such that the corresponding Newton--Okounkov body $\Delta(X(w), \mathcal{L}_\lambda, v_{{\mathbf s}_{\bm i} ^{\rm mut}}, \tau_\lambda)$ is related to the Nakashima--Zelevinsky polytope $\widetilde{\Delta}_{\bm i} (\lambda)$ by a unimodular transformation.
\end{enumerate}
\end{thm}

\begin{rem}
We can extend \cref{t:NO_body_parametrized_by_seeds} (4), (5) to non-simply-laced case by taking refinements of the opposite dominance orders $\preceq_{\varepsilon^{\bm i}} ^{\rm op}$, $\preceq_{\varepsilon^{\bm{i}, {\rm mut}}} ^{\rm op}$ appropriately (see \cite[Proposition 6.4 and Theorem 6.24]{FO2}).
\end{rem}

\begin{ex}[{see \cite[Example 7.12]{FO2}}]
Let $G = SL_4 (\c)$, and ${\bm i} = (1, 2, 1, 3, 2, 1) \in R(w_0)$.
Then we have 
\[M_{\bm i} = \begin{pmatrix}
1 & 0 & 0 & 0 & 0 & 0 \\
1 & 1 & 0 & 0 & 0 & 0 \\
0 & 1 & 1 & 0 & 0 & 0 \\
1 & 1 & 0 & 1 & 0 & 0 \\
0 & 1 & 1 & 1 & 1 & 0 \\
0 & 0 & 0 & 1 & 1 & 1 
\end{pmatrix},\]
and the Newton--Okounkov body $\Delta(G/B, \mathcal{L}_\lambda, v_{{\bf s}_{\bm i}}, \tau_\lambda)$ coincides with the set of $(g_1, \ldots, g_6) \in \r^6$ satisfying the following inequalities:
\begin{align*}
&0 \leq g_6 \leq \langle \lambda, h_1 \rangle,\ 0 \leq g_5 \leq \langle \lambda, h_2 \rangle,\ 0 \leq g_4 \leq \langle \lambda, h_3 \rangle,\ -g_5 \leq g_3 \leq -g_6 + \langle \lambda, h_1 \rangle,\\
&-g_4 \leq g_2 \leq -g_5 + \langle \lambda, h_2 \rangle,\ -g_2 -g_4 \leq g_1 \leq -g_3 -g_6 + \langle \lambda, h_1 \rangle.
\end{align*}
\end{ex}

In addition, by \cite[Corollary 7.7 (3)]{FO2} and \cref{t:unique_lattice_point_Steinert}, we obtain the following.

\begin{cor}\label{c:unique_lattice_point_reflexivity}
If $G$ is simply-laced and $w = w_0$, then the following hold for all $t \in \mathbb{T}$.
\begin{enumerate}
\item[{\rm (1)}] The Newton--Okounkov body $\Delta(G/B, \mathcal{L}_{2 \rho}, v_{\mathbf{s}_t}, \tau_{2 \rho})$ contains exactly one lattice point ${\bm a}_t$ in its interior.
\item[{\rm (2)}] The dual $\Delta(G/B, \mathcal{L}_{2 \rho}, v_{\mathbf{s}_t}, \tau_{2 \rho})^\vee$ is a lattice polytope.
\end{enumerate}
\end{cor}

\subsection{Combinatorial mutations on Newton--Okounkov bodies}\label{ss:combinatorial_mutations_on_NO}

Recall from \cref{c:unique_lattice_point_reflexivity} that ${\bm a}_t$ denotes the unique interior lattice point of $\Delta(G/B, \mathcal{L}_{2 \rho}, v_{\mathbf{s}_t}, \tau_{2 \rho})$. As an application of \cref{t:NO_body_parametrized_by_seeds} and \cref{p:tropicalized_muation_as_combinatorial}, let us prove the following.

\begin{thm}\label{t:combinatorial_mutations_NO}
If $G$ is simply-laced, then the following hold.
\begin{enumerate}
\item[{\rm (1)}] For fixed $w \in W$ and $\lambda \in P_+$, the Newton--Okounkov bodies $\Delta(X(w), \mathcal{L}_\lambda, v_{\mathbf{s}_t}, \tau_\lambda)$, $t \in \mathbb{T}$, are all combinatorially mutation equivalent in $M_\r$ up to unimodular transformations.
\item[{\rm (2)}] For $w = w_0$ and $\lambda = 2\rho$, the translated polytopes $\Delta(G/B, \mathcal{L}_{2\rho}, v_{\mathbf{s}_t}, \tau_{2\rho}) -{\bm a}_t$, $t \in \mathbb{T}$, are all combinatorially mutation equivalent in $M_\r$ up to unimodular transformations. 
As a consequence, the dual polytopes $\Delta(G/B, \mathcal{L}_{2 \rho}, v_{\mathbf{s}_t}, \tau_{2 \rho})^\vee$, $t \in \mathbb{T}$, are all combinatorially mutation equivalent in $N_\r$ up to unimodular transformations.
\end{enumerate}
\end{thm}

\cref{t:combinatorial_mutations_NO} (1) directly follows from \cref{t:NO_body_parametrized_by_seeds} and \cref{p:tropicalized_muation_as_combinatorial}. 
In order to prove \cref{t:combinatorial_mutations_NO} (2), we compute ${\bm a}_t$ explicitly.

\begin{prop}\label{p:unique_lattice_point_NO_cluster}
If $G$ is simply-laced, then the unique interior lattice point ${\bm a}_t = (a_j)_{j \in J}$ of $\Delta(G/B, \mathcal{L}_{2 \rho}, v_{\mathbf{s}_t}, \tau_{2 \rho})$ is given by
\[
a_j = 
\begin{cases}
0 &({\rm if}\ j \in J_{\rm uf}),\\
1 &({\rm if}\ j \in J_{\rm fr})
\end{cases}
\]
for $j \in J$. In particular, ${\bm a}_t$ is independent of the choice of $t \in \mathbb{T}$, and fixed under the tropicalized cluster mutations.
\end{prop}

\begin{proof}
Since $\mathfrak{g}$ is isomorphic to a direct sum of simply-laced simple Lie algebras as a Lie algebra, there exists ${\bm i}_{\mathfrak{g}} \in R(w_0)$ which is a concatenation of reduced words ${\bm i}_{X_n}$ defined in \cref{ex:unique_lattice_point}. Combining the computation of ${\bm a}_{X_n}$ in \cref{ex:unique_lattice_point} with \cref{t:NO_body_parametrized_by_seeds} (4), we deduce the assertion for ${\mathbf s}_{t_0} = {\mathbf s}_{{\bm i}_{\mathfrak{g}}}$. We proceed by induction on the distance from $t_0$ in $\mathbb{T}$. Take $t, t^\prime \in \mathbb{T}$ and $k \in J_{\rm uf}$ such that $t \overset{k}{\text{---}} t^\prime$. We assume that the assertion holds for $t$. By definition, the tropicalized cluster mutation $\mu_k ^T$ is given by a unimodular transformation on each of the half spaces $\{(g_j)_{j \in J} \in \r^J \mid g_k \geq 0\}$ and $\{(g_j)_{j \in J} \in \r^J \mid g_k \leq 0\}$. In addition, $\mu_k ^T$ is identity on the boundary hyperplane $\{(g_j)_{j \in J} \in \r^J \mid g_k = 0\}$ which includes the interior lattice point ${\bm a}_t$ of $\Delta(G/B, \mathcal{L}_{2 \rho}, v_{\mathbf{s}_t}, \tau_{2 \rho})$. From these, we deduce that $\mu_k ^T({\bm a}_t)$ is an interior lattice point of 
\[\mu_k ^T (\Delta(G/B, \mathcal{L}_{2 \rho}, v_{\mathbf{s}_t}, \tau_{2 \rho})) = \Delta(G/B, \mathcal{L}_{2 \rho}, v_{\mathbf{s}_{t^\prime}}, \tau_{2 \rho}).\]
This implies the assertion for $t^\prime$, which proves the proposition.
\end{proof}

By \cref{t:NO_body_parametrized_by_seeds} (4) and \cref{p:unique_lattice_point_NO_cluster}, we can compute the unique interior lattice point of the string polytope $\Delta_{\bm i} (2 \rho)$ as follows.

\begin{cor}\label{c:unique_interior_string_polytope}
Let ${\bm i} = (i_1, \ldots, i_m) \in R(w_0)$. If $G$ is simply-laced, then the unique interior lattice point ${\bm a}_{\bm i} = (a_j)_{j \in J}$ of the string polytope $\Delta_{\bm i} (2 \rho)$ is given by
\[
a_j = \sum_{k \in J_{\rm fr};\ j \leq k} \langle s_{i_{j +1}} \cdots s_{i_k} \varpi_{i_k}, h_{i_j}\rangle
\]
for $j \in J$.
\end{cor}

\begin{proof}[{Proof of \cref{t:combinatorial_mutations_NO} (2)}]
We write $\Delta \coloneqq \Delta(G/B, \mathcal{L}_{2\rho}, v_{\mathbf{s}_t}, \tau_{2\rho})$, and set $\Delta' \coloneqq \mu_k^T(\Delta)$. 
Let us consider the unique interior lattice point ${\bm a}_t=(a_j)_{j \in J}$ of $\Delta$. 
Recall from \cref{p:tropicalized_muation_as_combinatorial} that $\mu_k^T = f \circ \varphi_{w,F}$ for specific $w \in M$, $F \subseteq w^\perp$, and $f \in GL_J (\mathbb{Z})$.
Since we have $a_j = 0$ for all $j \in J_{\rm uf}$ by \cref{p:unique_lattice_point_NO_cluster}, the definitions of $\varphi_{w, F}$ and $f$ imply that $\varphi_{w,F}({\bm a}_t)=f({\bm a}_t)={\bm a}_t$. 

In addition, we have 
\[\varphi_{w,F}(\Delta - {\bm a}_t) = f^{-1}(\Delta' - {\bm a}_t),\]
which implies by Propositions \ref{prop:compatibility}, \ref{p:compatibility_well-defined} that 
\begin{align*}
{\rm mut}_w(\Delta^\vee,F)^\ast = \varphi_{w,F}(\Delta - {\bm a}_t) = f^{-1}(\Delta' - {\bm a}_t);
\end{align*}
here, we note that $(\Delta^\vee)^\ast = \Delta - {\bm a}_t$. Hence it follows that
\[{\rm mut}_w(\Delta^\vee,F) = (f^{-1}(\Delta' - {\bm a}_t))^\ast.\]
In general, for $Q \subseteq M_\r$ containing the origin in its interior and $\gamma \in {\it GL}(M_\r)$, it follows from the definition of the polar dual that the equality $\gamma(Q)^\ast = \!^t\gamma^{-1}(Q^\ast)$ holds, where $\!^t\gamma \in {\it GL}(N_\r)$ denotes the dual map of $\gamma$. 
Indeed, we have
\begin{align*}
\gamma(Q)^\ast &= \{ v \in N_\r \mid \langle \gamma(u), v \rangle \geq -1 \ {\rm for\ all}\ u \in Q\} \\ 
&= \{ v \in N_\r \mid \langle u, \!^t\gamma(v) \rangle \geq -1 \ {\rm for\ all}\ u \in Q\} \\
&=\{ \ \!^t\gamma^{-1}(v) \mid v \in N_\r,\ \langle u, v \rangle \geq -1 \ {\rm for\ all}\ u \in Q\} \\
&=\!^t\gamma^{-1}(Q^\ast). 
\end{align*}
Notice that if $\gamma$ is unimodular, then so is $\!^t\gamma$. 

Hence we conclude that 
$${\rm mut}_w(\Delta^\vee,F) = \!^tf(\Delta'^\vee).$$
This implies the required assertion since $\!^tf$ is a unimodular transformation. 
\end{proof}

\begin{rem}
After this paper and \cite{FO2}, an interesting preprint \cite{Qin2} of Qin appeared, which shows that the upper global basis is a common triangular basis when the Cartan matrix of $\mathfrak{g}$ is symmetrizable. 
Using this result by Qin, we can generalize \cref{t:NO_body_parametrized_by_seeds} to non-simply-laced case; this generalization will be explained in the revised version of \cite{FO2}. 
Hence \cref{t:combinatorial_mutations_NO} and \cref{c:unique_interior_string_polytope} can also be extended to non-simply-laced case.
\end{rem}

\bigskip

\section{Relation with FFLV polytopes}

In this section, we restrict ourselves to the cases $G = SL_{n+1}(\c)$ (of type $A_n$) and $G = Sp_{2n}(\c)$ (of type $C_n$).
FFLV polytopes were introduced by Feigin--Fourier--Littelmann \cite{FeFL1, FeFL2} and Vinberg \cite{Vin} to study PBW-filtrations of $V(\lambda)$. 
Kiritchenko \cite{Kir} proved that the FFLV polytope $FFLV(\lambda)$ of type $A_n$ coincides with the Newton--Okounkov body of $(G/B, \mathcal{L}_\lambda)$ associated with a valuation given by counting the orders of zeros/poles along a specific sequence of translated Schubert varieties. 
Feigin--Fourier--Littelmann \cite{FeFL3} realized the FFLV polytopes of types $A_n$ and $C_n$ as Newton--Okounkov bodies of $(G/B, \mathcal{L}_\lambda)$ using a different kind of valuation. 
Ardila--Bliem--Salazar \cite{ABS} gave an explicit bijective piecewise-affine map from the Gelfand--Tsetlin polytope $GT(\lambda)$ of type $A_n$ (resp., type $C_n$) to the FFLV polytope $FFLV(\lambda)$ of type $A_n$ (resp., type $C_n$) by generalizing Stanley's transfer map \cite{Sta} to marked poset polytopes; recall Examples \ref{ex:GT_polytopes_type_A} and \ref{ex:GT_polytopes_type_C} for the definition of $GT(\lambda)$.
In this section, we relate Ardila--Bliem--Salazar's transfer map with combinatorial mutations.

\subsection{Marked poset polytopes}

In this subsection, we recall the definition of Ardila--Bliem--Salazar's marked order polytopes and marked chain polytopes together with their transfer map \cite{ABS}. 

First, we recall what a marked poset is. Let $\widetilde{\Pi}$ be a poset equipped with a partial order $\prec$, and $A \subseteq \widetilde{\Pi}$ a subset of $\widetilde{\Pi}$ containing all minimal elements and maximal elements in $\widetilde{\Pi}$. 
Take a vector $\lambda=(\lambda_a)_{a \in A} \in \r^A$, called a \textit{marking}, such that $\lambda_a \leq \lambda_b$ whenever $a \prec b$ in $\widetilde{\Pi}$. We call the triple $(\widetilde{\Pi},A,\lambda)$ a \textit{marked poset}. 

\begin{defi}[{\cite[Definition 1.2]{ABS}}]
Work with the same notation as above. We set
\begin{align*}
&{\mathcal O}(\widetilde{\Pi},A,\lambda)\coloneqq\{(x_p)_{p \in \widetilde{\Pi} \setminus A} \in \r^{\widetilde{\Pi} \setminus A} \mid 
x_p \leq x_q \text{ if }p \prec q, \; \lambda_a \leq x_p \text{ if }a \prec p, \; x_p \leq \lambda_a \text{ if }p \prec a\}, \\
&{\mathcal C}(\widetilde{\Pi},A,\lambda)\coloneqq\{(x_p)_{p \in \widetilde{\Pi} \setminus A} \in \r^{\widetilde{\Pi} \setminus A}_{\ge 0} \mid \sum_{i=1}^kx_{p_i} \leq \lambda_b-\lambda_a \text{ if }a \prec p_1 \prec \cdots \prec p_k \prec b\}. 
\end{align*}
If $\lambda \in \z^A$, then those polytopes are lattice polytopes (\cite[Lemma 3.5]{ABS}). 
The polytope ${\mathcal O}(\widetilde{\Pi},A,\lambda)$ is called the \textit{marked order polytope}, and ${\mathcal C}(\widetilde{\Pi},A,\lambda)$ is called the \textit{marked chain polytope}. 
\end{defi}

\begin{rem}
Originally, the \textit{order polytope} ${\mathcal O}(\Pi)$ and the \textit{chain polytope} ${\mathcal C}(\Pi)$ of a poset $\Pi$ were introduced by Stanley \cite{Sta}. 
The notions of marked poset polytopes generalize those of ordinary poset polytopes. 
Indeed, given a poset $\Pi$, by setting $\widetilde{\Pi} \coloneqq \Pi \cup \{\hat{0},\hat{1}\}$, $A \coloneqq \{\hat{0},\hat{1}\}$, and $\lambda = (\lambda_{\hat{0}},\lambda_{\hat{1}}) \coloneqq (0,1)$, where $\hat{0}$ (resp., $\hat{1}$) is the new minimum (resp., maximum) element not belonging to $\Pi$, we see that the marked order polytope ${\mathcal O}(\widetilde{\Pi},A,\lambda)$ (resp., ${\mathcal C}(\widetilde{\Pi},A,\lambda)$) coincides with the ordinary order polytope ${\mathcal O}(\Pi)$ (resp., the ordinary chain polytope ${\mathcal C}(\Pi)$). 
\end{rem}

In \cite[Theorem 3.4]{ABS}, a piecewise-affine bijection $\widetilde{\phi}$ from ${\mathcal O}(\widetilde{\Pi},A,\lambda)$ to ${\mathcal C}(\widetilde{\Pi},A,\lambda)$ was constructed, which is called a \textit{transfer map}. The piecewise-affine map $\widetilde{\phi} \colon \r^{\widetilde{\Pi} \setminus A} \rightarrow \r^{\widetilde{\Pi} \setminus A}$, $(x_p)_p \mapsto (x_p^\prime)_p$, is defined as follows: 
\[x_p^\prime \coloneqq \min(\{x_p - x_{p'} \mid p' \lessdot p, \ p' \in \widetilde{\Pi} \setminus A\} \cup \{ x_p - \lambda_{p'} \mid p' \lessdot p, \ p' \in A\})\]
for $p \in \widetilde{\Pi} \setminus A$, where for $p,q \in \widetilde{\Pi}$, $q \lessdot p$ means that $p$ \textit{covers} $q$, that is, $q \prec p$ and there is no $q' \in \widetilde{\Pi} \setminus \{p,q\}$ with $q \prec q' \prec p$. 

Now, we recall a key notion which we will use in the proof of \cref{thm:marked_chain_order}, called \textit{marked chain-order polytopes}, introduced in \cite{FF} and developed in \cite{FFLP}. 
We remark that the original notion of marked chain-order polytopes is more general, but we restrict it for our purpose. 
Take a marked poset $(\widetilde{\Pi}, A, \lambda)$ and fix $\Pi' \subseteq \widetilde{\Pi} \setminus A$. 
We define ${\mathcal O}_{\Pi'}(\widetilde{\Pi}, A, \lambda)$ as follows: 
\begin{align*}
{\mathcal O}_{\Pi'}(\widetilde{\Pi}, A, \lambda) \coloneqq \{&(x_p)_{p \in \widetilde{\Pi} \setminus A} \in \r^{\widetilde{\Pi} \setminus A} \mid \ x_p \geq 0 \text{ for all } p \in \Pi', \\
&\sum_{i=1}^k x_{p_i} \leq y_b - y_a \text{ for }a \prec p_1 \prec \cdots \prec p_k \prec b \text{ with }p_i \in \Pi' \text{ and }a,b \in \widetilde{\Pi} \setminus \Pi'\}, 
\end{align*}
where for $c \in \widetilde{\Pi} \setminus \Pi'$, we set 
\[
y_c \coloneqq \begin{cases} \lambda_c &\text{ if }c \in A, \\ x_c &\text{ otherwise}. \end{cases}
\]
We can directly check that ${\mathcal O}_\emptyset(\widetilde{\Pi}, A, \lambda)={\mathcal O}(\widetilde{\Pi}, A, \lambda)$ and ${\mathcal O}_{\widetilde{\Pi} \setminus A}(\widetilde{\Pi}, A, \lambda)={\mathcal C}(\widetilde{\Pi}, A, \lambda)$. 
By taking $\Pi'$ with $\emptyset \subsetneq \Pi' \subsetneq \widetilde{\Pi} \setminus A$, we obtain an ``intermediate polytope'' between a marked order polytope and a marked chain polytope. 
It is proved in \cite[Proposition 2.4]{FFLP} that if $\lambda \in \z^A$, then ${\mathcal O}_{\Pi'}(\widetilde{\Pi}, A, \lambda)$ is a lattice polytope for every $\Pi' \subseteq \widetilde{\Pi} \setminus A$. 
Define a map $\widetilde{\phi}_{\Pi'} \colon \r^{\widetilde{\Pi} \setminus A} \rightarrow \r^{\widetilde{\Pi} \setminus A}$, $(x_p)_p \mapsto (x_p^\prime)_p$, by 
\begin{align}\label{eq:transfer_t}
&x_p^\prime \coloneqq  \begin{cases}
\min(\{x_p-x_{p'} \mid p' \lessdot p, \ p' \in \widetilde{\Pi} \setminus A\} \cup \{ x_p-\lambda_{p'} \mid p' \lessdot p, \ p' \in A\}) &\text{ if }p \in \Pi', \\
x_p &\text{ otherwise}
\end{cases}
\end{align}
for $p \in \widetilde{\Pi} \setminus A$. 
Notice that $\widetilde{\phi}_{\widetilde{\Pi}\setminus A}=\widetilde{\phi}$ and $\widetilde{\phi}_\emptyset={\rm id}$. 
In \cite[Theorem 2.1]{FFLP}, it is proved that the map $\widetilde{\phi}_{\Pi'}$ gives a piecewise-affine bijection from ${\mathcal O}(\widetilde{\Pi}, A, \lambda)$ to ${\mathcal O}_{\Pi'}(\widetilde{\Pi}, A, \lambda)$.

\subsection{Combinatorial mutation equivalence of marked poset polytopes}

The second named author proved in \cite[Theorem 4.1]{Hig} that the transfer map between ordinary poset polytopes can be described as a composition of combinatorial mutations in $M_\r$. 
We can generalize this result to marked poset polytopes under some conditions. 

We say that a poset $\widetilde{\Pi}$ is \textit{pure} if every maximal chain in $\widetilde{\Pi}$ has the same length. 
When $\widetilde{\Pi}$ is pure, all chains starting from a minimal element in $\widetilde{\Pi}$ and ending at $p$ have the same length for each $p \in \widetilde{\Pi}$. 
We denote by $r(p)$ the length of such chains. 

Let $(\widetilde{\Pi},A,\lambda)$ be a marked poset with $\lambda \in \z^A$. Assume that $\widetilde{\Pi}$ is pure, and that $\lambda$ satisfies $\lambda_a=\lambda_b$ for all $a,b \in A$ with $r(a)=r(b)$. 
Then there exists ${\bm u} = (u_p)_{p \in \widetilde{\Pi} \setminus A} \in {\mathcal O}(\widetilde{\Pi},A,\lambda) \cap \z^{\widetilde{\Pi} \setminus A}$ such that
\begin{equation}\label{eq:assumption}
\begin{split}
u_p &= u_{p'} \text{ for all }p,p' \in \widetilde{\Pi} \setminus A \text{ with } r(p)=r(p'), \text{ and }\\
u_p &= \lambda_a \text{ for all }p \in \widetilde{\Pi} \setminus A \text{ and }a \in A\text{ with }r(p)=r(a). 
\end{split}
\end{equation}

Let $\lambda^r$ denote the marking given by $(\lambda^r)_a = r(a)$ for $a \in A$. 
Then it is proved in \cite[Corollary 23]{FFP} that for a pure poset $\widetilde{\Pi}$, 
${\mathcal O}(\widetilde{\Pi},A,\lambda^r)$ (resp., ${\mathcal C}(\widetilde{\Pi},A,\lambda^r)$) contains a unique interior lattice point. 
Indeed, the unique interior lattice point $(r_p)_{p \in \widetilde{\Pi} \setminus A} \in {\mathcal O}(\widetilde{\Pi},A,\lambda^r)$ is given by $r_p=r(p)$ for all $p \in \widetilde{\Pi} \setminus A$, while the unique interior lattice point $(r_p')_{p \in \widetilde{\Pi} \setminus A} \in {\mathcal C}(\widetilde{\Pi},A,\lambda^r)$ is given by $r_p'=1$ for all $p \in \widetilde{\Pi} \setminus A$. 
We notice that $(r_p)_p$ satisfies \eqref{eq:assumption} and $\widetilde{\phi}((r_p)_p) = (r_p')_p$. Write
\begin{align*}
\overline{{\mathcal O}}(\widetilde{\Pi},A,\lambda^r)\coloneqq{\mathcal O}(\widetilde{\Pi},A,\lambda^r) - (r_p)_{p \in \widetilde{\Pi} \setminus A}, \;\text{ and }\;
\overline{{\mathcal C}}(\widetilde{\Pi},A,\lambda^r)\coloneqq{\mathcal C}(\widetilde{\Pi},A,\lambda^r)-(r_p')_{p \in \widetilde{\Pi} \setminus A}. 
\end{align*}
Namely, $\overline{{\mathcal O}}(\widetilde{\Pi},A,\lambda^r)$ (resp., $\overline{{\mathcal C}}(\widetilde{\Pi},A,\lambda^r)$) contains the origin as the unique interior lattice point. 

We regard polytopes appearing below as ones living in $M_\r$. 
\begin{thm}\label{thm:marked_chain_order}
Let $\widetilde{\Pi}$ be a pure poset. 
\begin{itemize}
\item[(1)] Let $(\widetilde{\Pi},A,\lambda)$ be a marked poset with $\lambda \in \z^A$ such that $\lambda$ satisfies $\lambda_a=\lambda_b$ for all $a,b \in A$ with $r(a)=r(b)$. 
Take a (not necessarily interior) lattice point ${\bm u} = (u_p)_{p \in \widetilde{\Pi}\setminus A}$ satisfying \eqref{eq:assumption}. 
Then the translated marked order polytope ${\mathcal O}(\widetilde{\Pi},A,\lambda) - {\bm u}$ and the translated marked chain polytope ${\mathcal C}(\widetilde{\Pi},A,\lambda) - \widetilde{\phi}({\bm u})$ are combinatorially mutation equivalent in $M_\r$. 
\item[(2)] Consider the marked poset $(\widetilde{\Pi},A,\lambda^r)$. Then $\overline{{\mathcal O}}(\widetilde{\Pi},A,\lambda^r)$ and $\overline{{\mathcal C}}(\widetilde{\Pi},A,\lambda^r)$ are combinatorially mutation equivalent in $M_\r$. 
\end{itemize}
\end{thm}
\begin{proof}

Since the assertion (2) directly follows from (1), we will prove the assertion (1). 

\noindent
{\bf The first step}.
For each $p \in \widetilde{\Pi} \setminus A$, set 
\begin{align*}
w_p &\coloneqq -{\bm e}_p, \\
F_p &\coloneqq {\rm conv}(\{-{\bm e}_{p'} \mid p' \lessdot p, \ p' \in \widetilde{\Pi} \setminus A\} \cup \{{\bf 0} \mid p' \lessdot p, \ p' \in A\}), \text{ and }\\
\varphi_p &\coloneqq \varphi_{w_p,F_p}. 
\end{align*}
Note that $F_p \subseteq w_p^\perp$. Then the direct computation shows the following: 
\begin{align*}
\varphi_q((x_p)_{p \in \widetilde{\Pi} \setminus A})&=(x_p)_p - \min\{\langle (x_p)_p, v \rangle \mid v \in F_q \} w_q \\
&=(x_p)_p + \min(\{ -x_{p'} \mid p' \lessdot q, \ p' \in \widetilde{\Pi} \setminus A\} \cup \{ 0 \mid p' \lessdot q, \ p' \in A\}){\bm e}_q,
\end{align*}
which implies that if we write $\varphi_q((x_p)_{p \in \widetilde{\Pi} \setminus A}) = (x_p^\prime)_{p \in \widetilde{\Pi} \setminus A}$, then we have 
\begin{align*}
x_p^\prime&=\begin{cases}
\min(\{x_p-x_{p'} \mid p' \lessdot p, \ p' \in \widetilde{\Pi} \setminus A\} \cup \{x_p \mid p' \lessdot p, \ p' \in A\}) &\text{ if }p=q, \\
x_p &\text{ otherwise}
\end{cases}
\end{align*}
for $p \in \widetilde{\Pi} \setminus A$.
We write $\widetilde{\Pi}\setminus A=\{q_1,\ldots,q_d\}$, and arrange $q_1,q_2,\ldots,q_d$ such as $q_i \prec q_j$ in $\widetilde{\Pi}$ only if $i>j$. 
Let $$\overline{\varphi}_i \coloneqq \varphi_{q_i} \circ \cdots \circ \varphi_{q_2} \circ \varphi_{q_1} \text{ for }i=1,\ldots,d.$$ 
In particular, $\overline{\varphi}_d$ includes all $\varphi_q$'s for $q \in \widetilde{\Pi}\setminus A$ arranged in the order ``from top to bottom''. 
Since each $\varphi_q$ changes only the $q$-th entry based on $p'$-th entries with $p' \lessdot q$, we obtain the following: if we write $\overline{\varphi}_i((x_p)_{p \in \widetilde{\Pi} \setminus A}) = (x_p^\prime)_{p \in \widetilde{\Pi} \setminus A}$, then it follows that
\begin{align*}
x_p^\prime=\begin{cases}
\min(\{x_p-x_{p'} \mid p' \lessdot p, \ p' \in \widetilde{\Pi} \setminus A\} \cup \{x_p \mid p' \lessdot p, \ p' \in A\}) &\text{ if }p \in \{q_1,\ldots,q_i\}, \\
x_p &\text{ otherwise}
\end{cases}
\end{align*}
for $p \in \widetilde{\Pi} \setminus A$.

\noindent
{\bf The second step}.
For $\Pi' \subseteq \widetilde{\Pi} \setminus A$, we define a translation map $f_{\Pi'}$ as follows: \begin{align*}
f_{\Pi'}  \colon \r^{\widetilde{\Pi} \setminus A} \rightarrow \r^{\widetilde{\Pi} \setminus A}, \ (x_p)_{p \in \widetilde{\Pi} \setminus A} \mapsto (x_p)_{p \in \widetilde{\Pi} \setminus A} -\widetilde{\phi}_{\Pi'}({\bm u}),
\end{align*}
where $\widetilde{\phi}_{\Pi'}$ is the map defined in \eqref{eq:transfer_t}. For simplicity, we write $f \coloneqq f_\emptyset$. 

In the third step, we will prove that 
\begin{align}\label{eq:goal}
\overline{\varphi}_i=f_{\{q_1,\ldots,q_i\}} \circ \widetilde{\phi}_{\{q_1,\ldots,q_i\}} \circ f^{-1} \text{ for all }i=1,\ldots,d. 
\end{align}
Once we prove this, we obtain that $\overline{\varphi}_i ({\mathcal O}(\widetilde{\Pi}, A, \lambda)-{\bm u})$ is convex for all $i$ since 
\begin{align*}
\overline{\varphi}_i ({\mathcal O}(\widetilde{\Pi}, A, \lambda)-{\bm u})&=f_{\{q_1,\ldots,q_i\}} \circ \widetilde{\phi}_{\{q_1,\ldots,q_i\}} \circ f^{-1}({\mathcal O}(\widetilde{\Pi}, A, \lambda)-{\bm u}) \\
&=f_{\{q_1,\ldots,q_i\}} \circ \widetilde{\phi}_{\{q_1,\ldots,q_i\}}({\mathcal O}(\widetilde{\Pi}, A, \lambda)) \\
&=f_{\{q_1,\ldots,q_i\}}({\mathcal O}_{\{q_1,\ldots,q_i\}}(\widetilde{\Pi}, A, \lambda)),
\end{align*}
and ${\mathcal O}_{\{q_1,\ldots,q_i\}}(\widetilde{\Pi}, A, \lambda)$ is a lattice polytope by \cite[Proposition 2.4]{FFLP}. Moreover, we see that
\begin{align*}
\overline{\varphi}_d ({\mathcal O}(\widetilde{\Pi}, A, \lambda)-{\bm u}) =f_{\widetilde{\Pi} \setminus A}({\mathcal C}(\widetilde{\Pi}, A, \lambda))={\mathcal C}(\widetilde{\Pi}, A, \lambda)-\widetilde{\phi}({\bm u}), 
\end{align*}
as required. 

\noindent
{\bf The third step}. We prove \eqref{eq:goal}. Given $(x_p)_{p \in \widetilde{\Pi} \setminus A} \in \r^{\widetilde{\Pi} \setminus A}$, 
we apply the map $\widetilde{\phi}_{\{q_1,\ldots,q_i\}}$ to $f^{-1}((x_p)_p)=(x_p + u_p)_p$. If we write $\widetilde{\phi}_{\{q_1,\ldots,q_i\}}((x_p + u_p)_p) = (x_p^\prime)_p$, then it holds that
\begin{align*}
x_p^\prime &= \begin{cases}
\min(\{x_p - x_{p'} + u_p - u_{p'} \mid p' \lessdot p, \ p' \in \widetilde{\Pi} \setminus A\} \cup \{x_p + u_p - \lambda_{p'} \mid p' \lessdot p, \ p' \in A\}) \\
\quad\quad\quad\;\quad \text{if } p \in \{q_1,\ldots,q_i\}, \\
x_p + u_p \;\;\; \text{ otherwise}
\end{cases}
\end{align*}
for $p \in \widetilde{\Pi} \setminus A$.
Remark that the set $\{u_p - u_{p'} \mid p' \lessdot p, \ p' \in \widetilde{\Pi} \setminus A\} \cup \{u_p - \lambda_{p'} \mid p' \lessdot p, \ p' \in A\}$ consists of only one element by the assumption \eqref{eq:assumption}. Hence we see that 
\begin{align*}
x_p^\prime &= \begin{cases}
\min(\{x_p - x_{p'} \mid p' \lessdot p, \ p' \in \widetilde{\Pi} \setminus A\} \cup \{x_p \mid p' \lessdot p, \ p' \in A\}) + \widetilde{\phi}_{\{q_1,\ldots,q_i\}}({\bm u})_p \\
\quad\quad\quad\quad\quad\quad\quad\quad\quad\quad \text{if } p \in \{q_1,\ldots,q_i\}, \\
x_p+\widetilde{\phi}_{\{q_1,\ldots,q_i\}}({\bm u})_p \;\;\;\quad \text{otherwise}
\end{cases} 
\end{align*}
for $p \in \widetilde{\Pi} \setminus A$.
Let us apply $f_{\{q_1,\ldots,q_i\}}$ to $\widetilde{\phi}_{\{q_1,\ldots,q_i\}} \circ f^{-1} ((x_p)_p) = (x_p^\prime)_p$. If we write 
\[
f_{\{q_1,\ldots,q_i\}} \circ \widetilde{\phi}_{\{q_1,\ldots,q_i\}} \circ f^{-1} ((x_p)_{p \in \widetilde{\Pi} \setminus A}) = f_{\{q_1,\ldots,q_i\}} ((x_p^\prime)_{p \in \widetilde{\Pi} \setminus A}) = (x'' _p)_{p \in \widetilde{\Pi} \setminus A},
\] 
then it holds for $p \in \widetilde{\Pi} \setminus A$ that 
\begin{align*}
x'' _p &=\begin{cases}
\min(\{x_p - x_{p'} \mid p' \lessdot p, \ p' \in \widetilde{\Pi} \setminus A \} \cup \{x_p \mid p' \lessdot p, \ p' \in A\}) &\text{ if } p \in \{q_1,\ldots,q_i\}, \\
x_p &\text{ otherwise}. 
\end{cases}
\end{align*}
Combining this with the first step, we conclude the desired equality \eqref{eq:goal}. 
\end{proof}

As the following example shows, the transfer map $\widetilde{\phi}$ is not necessarily described as a composition of combinatorial mutations in $M_\r$ if we drop the assumption \eqref{eq:assumption}. 

\begin{ex}\label{ex:counter_example}
Let us consider the pure marked poset in Figure \ref{marked_Hasse_counter_example} with a marking $\lambda = (\lambda_1, \ldots, \lambda_4) \in \z^4$ such that $\lambda_1 \leq \lambda_2 \leq \lambda_4$ and $\lambda_1 \leq \lambda_3 \leq \lambda_4$. 
We regard the marked poset polytopes as living in $\r^3$.
If $\lambda_2 = \lambda_3$, then the marked poset satisfies the assumption in \cref{thm:marked_chain_order}; hence we can apply \cref{thm:marked_chain_order} to the associated marked poset polytopes.
However, if $\lambda_2 \neq \lambda_3$, then this is not the case. 
For instance, if $\lambda_2 < \lambda_3$, then the transfer map $\widetilde{\phi} \colon \r^3 \rightarrow \r^3$ is given as follows: 
\begin{align}
\widetilde{\phi}(x,y,z) &= (\min\{x-\lambda_2,x-z\},\min\{y-z,y-\lambda_3\},z-\lambda_1)\notag\\
&= \begin{cases}
(x-z,y-z,z-\lambda_1) &\text{ if }z \geq \lambda_3, \\
(x-z,y-\lambda_3,z-\lambda_1) &\text{ if }\lambda_2 \leq z \leq \lambda_3, \\
(x-\lambda_2,y-\lambda_3,z-\lambda_1) &\text{ if }z \leq \lambda_2. 
\end{cases}\label{eq:three_affine_functions}
\end{align}
Hence, even if we apply any translation to the marked order polytope, the transfer map never becomes piecewise-linear. 
More precisely, we cannot make all the three affine maps $\r^3 \rightarrow \r^3$ in \eqref{eq:three_affine_functions} fix the origin.
This implies that the transfer map $\widetilde{\phi}$ cannot be described as a composition of combinatorial mutations in $M_\r$.

\begin{figure}[!ht]
\begin{center}
   \includegraphics[width=5.0cm,bb=80mm 190mm 130mm 230mm,clip]{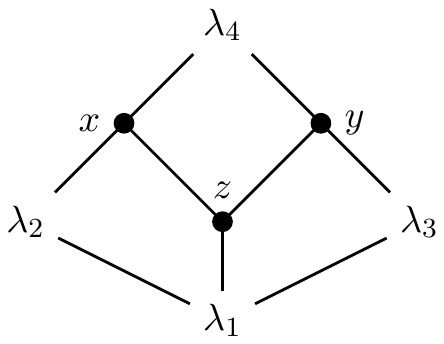}
	\caption{The marked Hasse diagram considered in \cref{ex:counter_example}.}
	\label{marked_Hasse_counter_example}
\end{center}
\end{figure}
\end{ex}

\subsection{Type $A$ case}

Let $G = SL_{n+1} (\c)$, and $\lambda \in P_+$. We write $\lambda_{\geq k} \coloneqq \sum_{k \leq \ell \leq n} \langle \lambda, h_\ell \rangle$ for $1 \leq k \leq n$. Let $\mathcal{O}_\lambda$ (resp., $\mathcal{C}_\lambda$) denote the marked order (resp., chain) polytope associated with a marked poset whose Hasse diagram is given in Figure \ref{type_A_marked_Hasse}.

\begin{figure}[!ht]
\begin{center}
   \includegraphics[width=10.0cm,bb=40mm 120mm 170mm 230mm,clip]{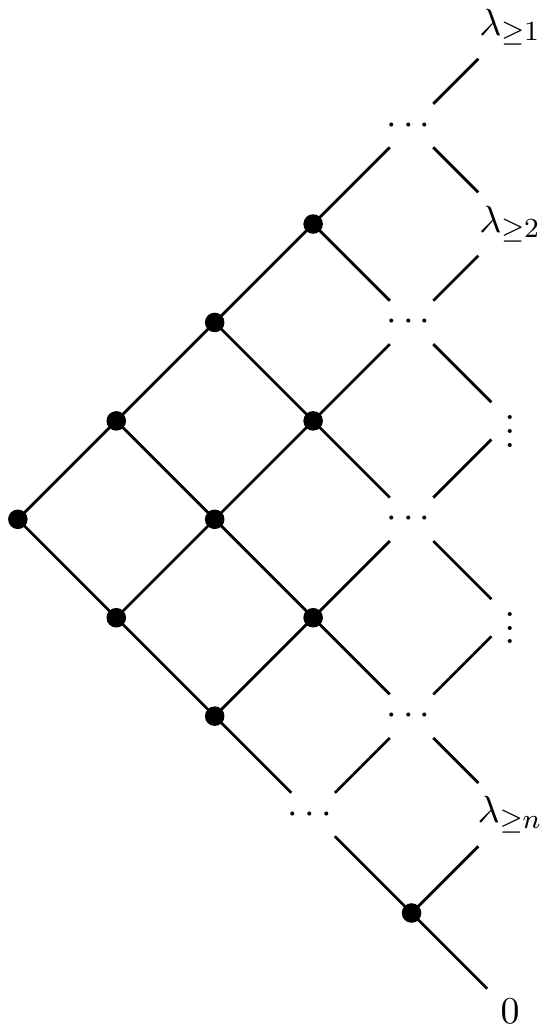}
	\caption{The marked Hasse diagram in type $A_n$.}
	\label{type_A_marked_Hasse}
\end{center}
\end{figure}

By the definition, the marked order polytope $\mathcal{O}_\lambda$ coincides with the Gelfand--Tsetlin polytope $GT(\lambda)$ (see \cref{ex:GT_polytopes_type_A}), and the marked chain polytope $\mathcal{C}_\lambda$ coincides with the FFLV polytope $FFLV(\lambda)$ (see \cite[equation (0.1)]{FeFL1}). 

Since the associated marked poset satisfies the assumption in \cref{thm:marked_chain_order}, the following theorem is an immediate consequence of \cref{thm:marked_chain_order}. 

\begin{thm}\label{t:relation_with_FFLV_type_A}
The following hold.
\begin{enumerate}
\item[{\rm (1)}] For all $\lambda \in P_+$, the polytopes $GT(\lambda)-{\bm u}$ and $FFLV(\lambda)-\widetilde{\phi}({\bm u})$ are combinatorially mutation equivalent in $M_\r$, where ${\bm u} \in GT(\lambda)$ is a lattice point satisfying \eqref{eq:assumption}. 
\item[{\rm (2)}] The dual polytopes $GT(2\rho)^\vee$ and $FFLV(2\rho)^\vee$ are combinatorially mutation equivalent in $N_\r$.
\end{enumerate}
\end{thm}

\subsection{Type $C$ case}

Let $G = Sp_{2n} (\c)$, and $\lambda \in P_+$. We write $\lambda_{\leq k} \coloneqq \sum_{1 \leq \ell \leq k} \langle \lambda, h_\ell \rangle$ for $1 \leq k \leq n$. Let $\mathcal{O}_\lambda$ (resp., $\mathcal{C}_\lambda$) denote the marked order (resp., chain) polytope associated with a marked poset whose Hasse diagram is given in Figure \ref{type_C_marked_Hasse}.

\begin{figure}[!ht]
\begin{center}
   \includegraphics[width=10.0cm,bb=40mm 140mm 170mm 230mm,clip]{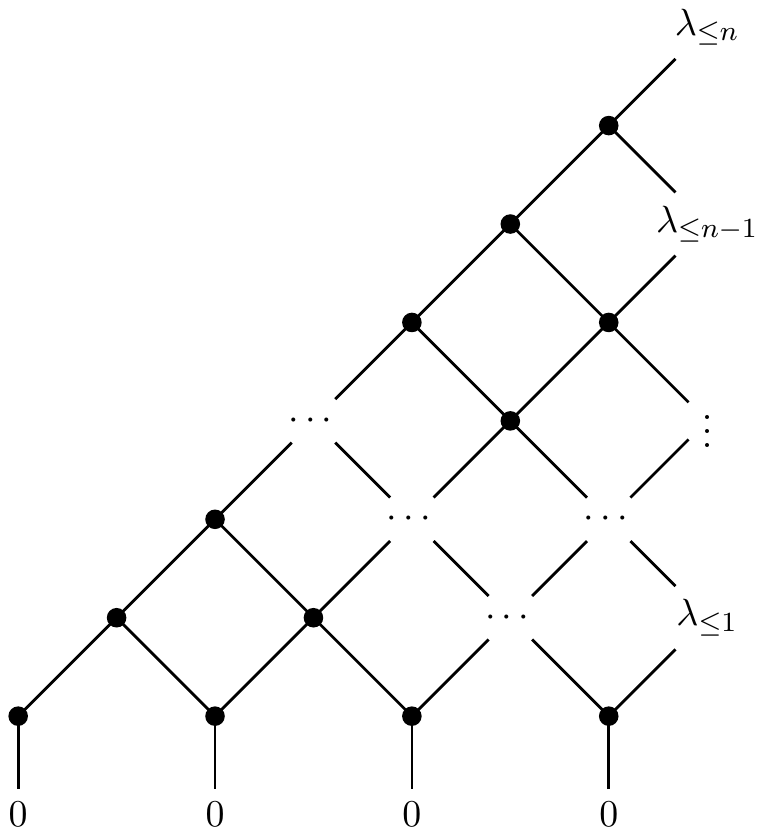}
	\caption{The marked Hasse diagram in type $C_n$.}
	\label{type_C_marked_Hasse}
\end{center}
\end{figure}

By the definition, the marked order polytope $\mathcal{O}_\lambda$ coincides with the Gelfand--Tsetlin polytope $GT_{C_n}(\lambda)$ of type $C_n$ (see \cref{ex:GT_polytopes_type_C}), and the marked chain polytope $\mathcal{C}_\lambda$ coincides with the FFLV polytope $FFLV_{C_n}(\lambda)$ of type $C_n$ (see \cite[equation (1.2)]{FeFL2}).

Similar to \cref{t:relation_with_FFLV_type_A}, the following theorem holds. 

\begin{thm}\label{t:relation_with_FFLV_type_C}
The following hold.
\begin{enumerate}
\item[{\rm (1)}] For all $\lambda \in P_+$, the polytopes $GT_{C_n}(\lambda) - {\bm u}$ and $FFLV_{C_n}(\lambda) - \widetilde{\phi}({\bm u})$ are combinatorially mutation equivalent in $M_\r$, where ${\bm u} \in GT_{C_n}(\lambda)$ is a lattice point satisfying \eqref{eq:assumption}. 
\item[{\rm (2)}] The dual polytopes $GT_{C_n}(2\rho)^\vee$ and $FFLV_{C_n}(2\rho)^\vee$ are combinatorially mutation equivalent in $N_\r$.
\end{enumerate}
\end{thm}


\bigskip

\vspace{4mm}
\begin{thebibliography}{99999}
\bibitem{ACGK}
M. Akhtar, T. Coates, S. Galkin, and A. M. Kasprzyk, Minkowski polynomials and mutations, SIGMA Symmetry Integrability Geom.\ Methods Appl.\ {\bf 8} (2012), 094, 17 pages.
\bibitem{And}
D. Anderson, Okounkov bodies and toric degenerations, Math.\ Ann.\ {\bf 356} (2013), no.~3, 1183--1202.
\bibitem{ABS}
F. Ardila, T. Bliem, and D. Salazar, Gelfand--Tsetlin polytopes and Feigin--Fourier--Littelmann--Vinberg polytopes as marked poset polytopes, J. Combin.\ Theory Ser.\ A {\bf 118} (2011), no.~8, 2454--2462.
\bibitem{BFZ}
A. Berenstein, S. Fomin, and A. Zelevinsky, Cluster algebras. III. Upper bounds and double Bruhat cells, Duke Math.\ J. {\bf 126} (2005), no.\ 1, 1--52.
\bibitem{BZ}
A. Berenstein and A. Zelevinsky, Tensor product multiplicities, canonical bases and totally positive varieties, Invent.\ Math.\ {\bf 143} (2001), no.\ 1, 77--128.
\bibitem{Bri}
M. Brion, Lectures on the geometry of flag varieties, in Topics in Cohomological Studies of Algebraic Varieties, Trends Math., Birkh\"{a}user, Basel, 2005, 33--85.
\bibitem{CCGGK}
T. Coates, A. Corti, S. Galkin, V. Golyshev, and A. M. Kasprzyk, Mirror symmetry and Fano manifolds, in European Congress of Mathematics, Eur.~Math.~Soc., Z\"{u}rich, 2013, 285--300.
\bibitem{EH}
L. Escobar and M. Harada, Wall-crossing for Newton-Okounkov bodies and the tropical Grassmannian, with Appendix by N. Ilten, preprint 2019, arXiv:1912.04809v1.
\bibitem{FF}
X. Fang and G. Fourier, Marked chain-order polytopes, European J. Combin.~{\bf 58} (2016), 267--282. 
\bibitem{FFLP}
X. Fang, G. Fourier, J.-P. Litza, and C. Pegel, A continuous family of marked poset polytopes, to appear in SIAM Journal on Discrete Mathematics. 
\bibitem{FFP}
X. Fang, G. Fourier, and C. Pegel, The Minkowski property and reflexivity of marked poset polytopes, Electron.\ J. Combin.\ {\bf 27} (2020), no.\ 1, P1.27.
\bibitem{FeFL1}
E. Feigin, G. Fourier, and P. Littelmann, PBW filtration and bases for irreducible modules in type $A_n$, Transform.\ Groups {\bf 16} (2011), no.~1, 71--89.
\bibitem{FeFL2}
E. Feigin, G. Fourier, and P. Littelmann, PBW filtration and bases for symplectic Lie algebras, Int.\ Math.\ Res.\ Not.\ IMRN {\bf 2011} (2011), no.~24, 5760--5784.
\bibitem{FeFL3}
E. Feigin, G. Fourier, and P. Littelmann, Favourable modules: filtrations, polytopes, Newton--Okounkov bodies and flat degenerations, Transform.\ Groups {\bf 22} (2017), no.~2, 321--352.
\bibitem{FG:amal}
V. V. Fock and A. B. Goncharov, Cluster {$\mathscr X$}-varieties, amalgamation, and {P}oisson-{L}ie groups, in Algebraic Geometry and Number Theory, 
Progr.\ Math.\ Vol.\ 253, Birkh\"{a}user Boston, Boston, MA, 2006, 27--68.
\bibitem{FG}
V. V. Fock and A. B. Goncharov, Cluster ensembles, quantization and the dilogarithm, Ann.\ Sci.\ \'Ec.\ Norm.\ Sup\'er.\ (4) {\bf 42} (2009), no.\ 6, 865--930.
\bibitem{FZ:ClusterI}
S. Fomin and A. Zelevinsky, Cluster algebras. I. Foundations, J.\ Amer.\ Math.\ Soc.\ {\bf 15} (2002), no.\ 2, 497--529.
\bibitem{FZ:ClusterIV}
S. Fomin and A. Zelevinsky, Cluster algebras. IV. {C}oefficients, Compos.\ Math.\ {\bf 143} (2007), no.\ 1, 112--164.
\bibitem{Fuj}
N. Fujita, Polyhedral realizations of crystal bases and convex-geometric Demazure operators, Selecta Math.\ (N.S.) {\bf 25} (2019), Paper No.\ 74.
\bibitem{FN}
N. Fujita and S. Naito, Newton--Okounkov convex bodies of Schubert varieties and polyhedral realizations of crystal bases, Math.\ Z.\ {\bf 285} (2017), no.~1-2, 325--352.
\bibitem{FO1}
N. Fujita and H. Oya, A comparison of Newton--Okounkov polytopes of Schubert varieties, J.~Lond.~Math.~Soc.~(2) {\bf 96} (2017), no.~1, 201--227.
\bibitem{FO2}
N. Fujita and H. Oya, Newton--Okounkov polytopes of Schubert varieties arising from cluster structures, preprint 2020, arXiv:2002.09912v1.
\bibitem{GHKK}
M. Gross, P. Hacking, S. Keel, and M. Kontsevich, Canonical bases for cluster algebras, J. Amer.\ Math.\ Soc.\ {\bf 31} (2018), no.\ 2, 497--608. 
\bibitem{HK}
M. Harada and K. Kaveh, Integrable systems, toric degenerations, and Okounkov bodies, Invent.\ Math.\ {\bf 202} (2015), no.~3, 927--985.
\bibitem{Hib}
T. Hibi, Dual polytopes of rational convex polytopes, Combinatorica {\bf 12} (1992), no.\ 2, 237--240.
\bibitem{Hig}
A. Higashitani, Two poset polytopes are mutation equivalent, preprint 2020, arXiv:2002.01364v1.
\bibitem{Ilt}
N. Ilten, Mutations of Laurent polynomials and flat families with toric fibers, SIGMA Symmetry Integrability Geom.\ Methods Appl.\ {\bf 8} (2012), 047, 7 pages.
\bibitem{Jan}
J. C. Jantzen, Representations of Algebraic Groups, 2nd ed., Math.\ Surveys Monographs Vol.\ 107, Amer.\ Math.\ Soc., Providence, RI, 2003.
\bibitem{Kas}
M. Kashiwara, On crystal bases, in Representations of Groups ({B}anff, {AB}, 1994), CMS Conf.\ Proc.\ Vol.~16, Amer.\ Math.\ Soc., Providence, RI, 1995, 155--197.
\bibitem{Kav}
K. Kaveh, Crystal bases and Newton--Okounkov bodies, Duke Math.\ J. {\bf 164} (2015), no.~13, 2461--2506.
\bibitem{KK1}
K. Kaveh and A. G. Khovanskii, Convex bodies and algebraic equations on affine varieties, preprint 2008, arXiv:0804.4095v1; a short version with title {\it Algebraic equations and convex bodies} appeared in Perspectives in Analysis, Geometry, and Topology, Progr.\ Math.~Vol.~296, Birkh\"{a}user/Springer, New York, 2012, 263--282.
\bibitem{KK2}
K. Kaveh and A. G. Khovanskii, Newton--Okounkov bodies, semigroups of integral points, graded algebras and intersection theory, Ann.\ of Math.\ {\bf 176} (2012), no.~2, 925--978.
\bibitem{Kir}
V. Kiritchenko, Newton--Okounkov polytopes of flag varieties, Transform.\ Groups {\bf 22} (2017), no.~2, 387--402.
\bibitem{Kum}
S. Kumar, Kac--Moody Groups, their Flag Varieties and Representation Theory, Progr.\ Math.~Vol.~204, Birkh\"{a}user Boston, Inc., Boston, MA, 2002.
\bibitem{LM}
R. Lazarsfeld and M. Mustata, Convex bodies associated to linear series, Ann.\ Sci.\ de I'ENS\ {\bf 42} (2009), no.~5, 783--835.
\bibitem{Lit}
P. Littelmann, Cones, crystals, and patterns, Transform.\ Groups {\bf 3} (1998), no.\ 2, 145--179.
\bibitem{Nak}
T. Nakashima, Polyhedral realizations of crystal bases for integrable highest weight modules, J.\ Algebra {\bf 219} (1999), no.~2, 571--597.
\bibitem{NZ}
T. Nakashima and A. Zelevinsky, Polyhedral realizations of crystal bases for quantized
Kac--Moody algebras, Adv.\ Math.\ {\bf 131} (1997), no.~1, 253--278.
\bibitem{Oko1}
A. Okounkov, Brunn--Minkowski inequality for multiplicities, Invent.\ Math.\ {\bf 125} (1996), no.~3, 405--411.
\bibitem{Oko2}
A. Okounkov, Multiplicities and Newton polytopes, in Kirillov's Seminar on Representation Theory, Amer.\ Math.\ Soc.\ Transl.\ Ser.\ 2 Vol.\ 181, Adv.\ Math.\ Sci.\ Vol.\ 35, Amer.\ Math.\ Soc., Providence, RI, 1998, 231--244.
\bibitem{Oko3}
A. Okounkov, Why would multiplicities be log-concave?, in The Orbit Method in Geometry and Physics, Progr.\ Math.~Vol.~213, Birkh\"{a}user Boston, Boston, 2003, 329--347.
\bibitem{Por}
I. Portakal, A note on deformations and mutations of fake weighted projective planes, in Algebraic and Geometric Combinatorics on Lattice Polytopes, World Scientific, 2019, 354--366.
\bibitem{Qin}
F. Qin, Triangular bases in quantum cluster algebras and monoidal categorification conjectures, Duke Math.\ J. {\bf 166} (2017), no.\ 12, 2337--2442.
\bibitem{Qin2}
F. Qin, Dual canonical bases and quantum cluster algebras, preprint 2020, arXiv:2003.13674v2.
\bibitem{Rus}
J. Rusinko, Equivalence of mirror families constructed from toric degenerations of flag varieties, Transform.\ Groups {\bf 13} (2008), no.\ 1, 173--194.
\bibitem{Sta}
R. P. Stanley, Two Poset Polytopes, Discrete Comput.\ Geom.\ {\bf 1} (1986), 9--23.
\bibitem{Ste}
C. Steinert, Reflexivity of Newton--Okounkov bodies of partial flag varieties, preprint 2019, arXiv:1902.07105v1.
\bibitem{Vin}
E. Vinberg, On some canonical bases of representation spaces of simple Lie
algebras, Conference Talk, Bielefeld, 2005.
\end{thebibliography}
\end{document}